\documentclass[a4paper, 10pt, twoside, notitlepage]{amsart}

\usepackage[utf8]{inputenc}
\usepackage{color}
\usepackage{amsmath} 
\usepackage{amssymb} 
\usepackage{amsthm}
\usepackage{geometry}
\usepackage{graphicx}
\usepackage{esint}
\usepackage[colorlinks=true,linkcolor=blue]{hyperref}

\theoremstyle{plain}
\newtheorem{thm}{Theorem}
\newtheorem{prop}{Proposition}[section]
\newtheorem{lem}[prop]{Lemma}

\newtheorem{defi}[prop]{Definition}
\newtheorem{rmk}[prop]{Remark}

\newtheorem{example}[prop]{Example}

\newcommand {\R} {\mathbb{R}} 
 \newcommand {\N} {\mathbb{N}}
 
\newcommand {\p} {\partial}
\newcommand {\dt} {\partial_t}

\newcommand {\dr} {\partial_r}

\newcommand {\D} {\Delta}

\newcommand {\sgn} {\text{sgn}}
\newcommand {\supp} {\text{supp}}

\DeclareMathOperator{\vol}{vol}

\title[Unique Continuation for Sublinear Elliptic Equations]{Unique Continuation for Sublinear Elliptic Equations Based on Carleman Estimates}

\author{Angkana R\"uland}
\address{Max-Planck-Institute for Mathematics in the Sciences, Inselstr. 22, 04103 Leipzig}
\email{rueland@mis.mpg.de}

\begin{document}

\maketitle

\begin{abstract}
In this article we deal with different forms of the unique continuation property for second order elliptic equations with nonlinear potentials of sublinear growth. Under suitable regularity assumptions, we prove the weak and the strong unique continuation property. Moreover, we also discuss the unique continuation property from measurable sets, which shows that nodal domains to these equations must have vanishing Lebesgue measure. Our methods rely on suitable Carleman estimates, for which we include the sublinear potential into the main part of the operator.
\end{abstract} 

\section{Introduction}

This article is devoted to unique continuation properties for second order elliptic equations with \emph{sublinear potentials}. The unique continuation property for second order elliptic equations has a long tradition and many important ramifications. In a sense, it quantifies the rigidity of solutions to these equations. 

More precisely, the \emph{(weak) unique continuation property} (WUCP) addresses the question of whether if a given solution $u$ to an equation $Lu=0$ in an open set $\Omega \subset \R^n$ vanishes in an open subset of $\Omega$, i.e. $u= 0$ in $B_{r}(x_0)\subset \Omega$, this already implies that $u$ vanishes globally in $\Omega$, i.e. whether $u\equiv 0$ in $\Omega$. 

Similarly, it is also possible to ask whether the \emph{vanishing of infinite order} at a point $x_0\in \Omega$ of a solution $u$ to $Lu=0$ in $\Omega$, i.e. whether the assumption that
\begin{align*}
\lim\limits_{r\rightarrow 0} r^{-m}\int\limits_{B_r(x_0)} u^2 dx = 0 \mbox{ for all } m \in \N,
\end{align*}
already implies the global vanishing of $u$ in $\Omega$. This property is referred to as the \emph{strong unique continuation property} (SUCP). 

In the context of nodal domain estimates of eigenfunctions to certain operators, also the \emph{unique continuation property from measurable sets} (MUCP) plays an important role, as it asserts that a solution $u$ of $Lu=0$ in $\Omega$, which vanishes on a set $E\subset \Omega$ of positive measure, already vanishes globally. In particular, if the MUCP holds for an equation, it implies that its nodal set has vanishing Lebesgue measure.

Prototypical examples of functions which satisfy all of these properties are harmonic functions (since they are analytic). However, this property holds for a much larger class of second order elliptic operators
\begin{align}
\label{eq:L}
Lu = \p_i a^{ij} \p_j u + W_{1,i} \p_i u + \p_i (W_{2,i} u) + V u
\end{align}
under suitable regularity assumptions on the uniformly elliptic metric $a^{ij}:\Omega \rightarrow \R^{n\times n}_+$, the gradient potentials $W_{1,i}, W_{2,i}: \Omega \rightarrow \R$ and the $L^2$ potential $V: \Omega \rightarrow \R$.
The setting of linear second order equations is by now quite well understood: Based on first works due to Carleman \cite{C39}, \emph{Carleman estimates} were developed as a major tool of addressing these problems. After early important results due to Aronszajn-Krzywicki-Szarski \cite{AKS62}, some of the seminal contributions in this context include the work due to Chanillo-Sawyer \cite{CS90}, Kenig-Ruiz-Sogge \cite{KRS87}, Jerison-Kenig \cite{JK85}, who deal with scaling critical potentials in different function spaces, e.g. $V\in L^{\frac{n}{2}}$, Wolff \cite{W90,W92}, who introduced osculation arguments in order to overcome intrinsic difficulties with Carleman estimates for low regularity gradient potentials (i.e. for potentials close to the critical space $W_{1,i}, W_{2,i} \in L^{n}$), c.f. \cite{J86}. Finally, Koch and Tataru \cite{KT01, KT05} showed how to combine Lipschitz continuous metrics $a^{ij}$ with critical function spaces for the gradient potentials $W_{1,i}, W_{2,i}$ and the $L^2$ potential $V$. Counterexamples \cite{P62, M74, M98, KT02, KN00, W92a, W94} show that both in the weak and the strong unique continuation setting these assumptions are essentially sharp. For a more extensive overview on the vast literature on unique continuation properties for second order elliptic equations, we refer to the survey article \cite{W93} and the above cited articles.

A second line of thought was introduced by Garofalo and Lin \cite{GL86, GL87}, who proved similar unique continuation results by means of variational arguments. Their main tool, which also found numerous applications in other variational problems such as for instance free boundary value problems, is a so-called \emph{frequency function}. This is used to measure the growth of solutions to \eqref{eq:L} away from its nodal set.

With the afore mentioned results at hand, it is also possible to study the unique continuation properties of a quite general class of second order \emph{semilinear} elliptic equations. Since in the linear theory the lower order terms (including the potentials $W_{1,i}, W_{2,i}, V$) are usually treated perturbatively, it is in particular possible to deal with equations with \emph{superlinear} potentials, the model problem being given by the equation
\begin{align*}
(-\D) u = |u|^{q-2}u, \ u \in L^{\infty}_{loc}(\Omega)\cap H^1_{loc}(\Omega), \ q\geq 2.
\end{align*}
Indeed, in this case, by setting $V= |u|^{q-2}$ and by using the assumed $L^{\infty}_{loc}$ property of $u$ (which is hence inherited by $V$), we may rewrite the problem in the form \eqref{eq:L}. 

Similarly, one can however also wonder whether analogous properties hold for \emph{sublinear} potentials, i.e. whether for instance solutions to the equation
\begin{align}
\label{eq:mod1}
(-\D) u = |u|^{q-2}u, \ u \in L^{\infty}_{loc}(\Omega)\cap H^1_{loc}(\Omega), \ q \in (1,2),
\end{align}
satisfy the various unique continuation properties from above. Here however, the setting changes -- not only because the previous reasoning of simply defining $V=|u|^{q-2}$ fails, since, in general, with this definition, the potential $V$ need no longer be a function in the space $L^{\frac{n}{2}}$. Indeed, in the \emph{sublinear} regime there are intrinsic difficulties which have to be overcome: Already when studying the related ODEs
\begin{align}
\label{eq:mod11}
u''=f_q(u), 
\end{align}
where $|f_q(u)|$ is bounded from below by $|u|^{q-1}$ for $q\in [1,2)$, one observes that in general the unique continuation property fails. For instance, a computation shows that for any $t_0\in \R$ the function
\begin{align*}
u(t)= \left\{ \begin{array}{ll}
\left(\frac{2q}{(2-q)^2}\right)^{\frac{1}{q-2}}(t-t_0)^{\frac{2}{2-q}} \mbox{ for } t>t_0,\\
0 \mbox{ for } t \leq t_0,
\end{array} \right.
\end{align*}
is a solution to the equation
\begin{align*}
u'' = |u|^{q-2}u, \ q\in(1,2).
\end{align*}

Motivated by the study of nonlinear eigenvalue problems, the analysis of the corresponding nodal domains \cite{PW15} and the relation of these problems to porous media type equations \cite{V07}, in a recent article Soave and Weth \cite{SW17} however observed that the right choice of the sign of the nonlinearity in \eqref{eq:mod11} allows one to recover the WUCP. For instance, direct energy methods and ODE arguments show that the solutions to the equation
\begin{align*}
u''=-|u|^{q-2}u, \ q \in (1,2)
\end{align*}
satisfy the UCP (these arguments even show that all zeros of $u$ must be simple zeros).
More generally, in arbitrary dimensions, Soave and Weth \cite{SW17} prove that considering correctly signed equations modelled on the problem 
\begin{align}
\label{eq:model0}
\begin{split}
-\D u = f_q(u) \mbox{ with } 
f_q(u) = \left\{
\begin{array}{ll}
|u|^{q-2}u \mbox{ if } q \in (1,2),\\
\sgn(u) \mbox{ if } q = 1,
\end{array}
\right.
\end{split}
\end{align}
it is possible to prove the WUCP. To this end, they adapt the frequency function techniques due to Garofalo and Lin \cite{GL86,GL87} and 
Garofalo and Smit Vega Garcia \cite{GSVG14}. In their work it however remained open, whether the SUCP and the MUCP hold. In particular, the corresponding estimates on the nodal domains for general sign changing solutions from \cite{PW15} remained open. 

In this article we address the unique continuation property for these equations by applying Carleman techniques. In particular, under suitable assumptions on the nonlinear potential, we also derive the SUCP and the MUCP, thus settling the question from \cite{PW15} (at least for $q\in (1,2)$).

\subsection{The results}
Let us discuss the precise results:
In order to motivate the problem and the ideas without having to deal with an additional layer of technicalities, in Section \ref{sec:model} we first address the \emph{constant coefficient setting} and explain the main ideas of our argument for a \emph{model situation}. Then, in Section \ref{sec:results_var} we generalize these results to \emph{variable coefficient equations} with more \emph{general sublinear potentials}. We remark that, as already observed by Soave-Weth \cite{SW17}, in both cases the sign of the sublinear nonlinearity is crucial. 

\subsubsection{The model case}
\label{sec:model}
In the sequel, as a model problem we consider a slight generalization of \eqref{eq:model0}. We seek to prove that solutions to this equation posses the (strong) unique continuation property as well as the unique continuation property from measurable sets:

\begin{thm}
\label{thm:SMUCP}
Let $\Omega \subset \R^n$ be open and let $x_0 \in \Omega$. Let $V\in L^{\infty}(\Omega)$. Suppose that $u\in H^1_{loc}(\Omega)\cap L^{\infty}_{loc}(\Omega)$ is a solution to
\begin{align}
\label{eq:model}
\D u + f_q(u)= V u \mbox{ in } \Omega \mbox{ with } 
f_q(u) = \left\{
\begin{array}{ll}
|u|^{q-2}u \mbox{ if } q \in (1,2),\\
\sgn(u) \mbox{ if } q = 1,
\end{array}
\right.
\end{align}
and assume that one of the following conditions holds:
\begin{itemize}
\item[(a)] $q\in [1,2)$ and there exists a radius $r_0>0$ such that $u$ vanishes on $B_{r_0}(x_0)\subset \Omega $.
\item[(b)] $q\in (1,2)$ and $u$ vanishes of infinite order at $x_0\in \Omega$, i.e. for any $m\in \N$ we have $\lim\limits_{r\rightarrow 0} r^{-m} \int\limits_{B_r(x_0)} u^2 dx = 0 $.
\item[(c)] $q\in (1,2)$ and $u$ vanishes on a measurable set of positive measure, i.e. there exists $E \subset \Omega$ such that $|E|>0$ and $u|_{E}=0$.
\end{itemize}
Then, $u \equiv 0$ in $\Omega$.
\end{thm}

We remark that the sign of $f_q(u)$ is crucial and that any form of the UCP is false if the sign of $f_q(u)$ were to be reversed (c.f. the discussion of ODE examples from above).

We prove the results of Theorem \ref{thm:SMUCP} by using Carleman estimates. In order to distinguish the setting with a favourable sign for the nonlinearity from the setting with an unfavourable sign, we include the nonlinearity into the leading part of the operator. This is in contrast to the situation of superlinear potentials, which one would typically treat perturbatively. 

\begin{thm}
\label{thm:sub_lin_Carl}
Let $q\in [1,2)$.
Let $u\in H^1_{loc}(\R^n)\cap L^{\infty}_{loc}(\R^n)$ be a solution of 
\begin{align}
\label{eq:eq_main}
\begin{split}
\D u + f_q(u) = g \mbox{ in } \R^n
\end{split}
\end{align}
with support contained in $B_1\setminus \overline{B_{r_1}}$ for some $r_1 \in (0,1/2)$.
Define $\phi(x):= \psi(|x|)$ to be
\begin{align*}
\psi(r) = -\ln(r) + \frac{1}{10}\left( \ln(r) \arctan(\ln(r))- \frac{1}{2} \ln(1+ \ln^2(r))  \right).
\end{align*}
Then there exist constants $\tau_0>1$ and $C>1$ (which only depend on $n,q$) such that for all $\tau \geq \tau_0$ we have
\begin{align*}
&\tau^{3/2} \|e^{\tau \phi} (1+\ln^2(|x|))^{-1/2} u\|_{L^2(\R^n)}
+ \tau^{1/2} \|e^{\tau \phi} |x| (1+\ln^2(|x|))^{-1/2} \nabla u \|_{L^2(\R^n)}\\
&+\left(\frac{q-2}{q}\right)^{1/2}\tau \|e^{\tau \phi} |x||u|^{\frac{q}{2}}\|_{L^2(\R^n)}
\leq C \|e^{\tau \phi}|x|^2 g\|_{L^2(\R^n)}.
\end{align*}
\end{thm}

\subsubsection{Variable coefficients and more general sublinear terms}
\label{sec:results_var}
More generally, with similar arguments, it is possible to treat the setting of more general nonlinearities $f_q(x,u)$ and equations, which involve Lipschitz metrics. In the sequel, we describe the assumptions precisely. For the metric $a^{ij}: \Omega \rightarrow \R^{n\times n}$ we assume that:
\begin{itemize}
\item[(A1)] The metric is uniformly elliptic, i.e. there exist constants $0< \lambda \leq 1 \leq \Lambda< \infty$ such that
\begin{align*}
\lambda |\xi|^2 \leq a^{ij}(x)\xi_i \xi_j \leq \Lambda |\xi|^2.
\end{align*} 
\item[(A2)] The metric is Lipschitz continuous, i.e. there exists a constant $\Lambda_0>1$ such that
\begin{align*}
|a^{ij}(x)-a^{ij}(y)|\leq \Lambda_0 |x-y| \mbox{ for } x,y \in \Omega.
\end{align*}
\item[(A3)] The metric is normalized, i.e. we assume that $0 \in \Omega$ and that $a^{ij}(0)=\delta_{ij}$, where $\delta_{ij}$ denotes the Kronecker symbol.
\end{itemize}
We remark that the assumption (A3) can always be imposed without any loss of generality, as it can always be achieved by a suitable affine change of coordinates.

Compared to the setting in \cite{SW17}, we in the sequel require stronger differentiability conditions for the nonlinearity. More precisely, for $q\in [1,2)$ we impose the following conditions:
\begin{itemize}
\item[(F1)] $0<s f_q(x,s)\leq q F_q(x,s)$ for $s\in (-\epsilon_0, \epsilon_0)\setminus \{0\}$, $x \in \Omega$, where $\epsilon_0>0$ is an arbitrary but fixed constant and where $f_q\in L^{\infty}_{loc}(\Omega \times \R)$ and $F_q: \Omega \times \R \rightarrow \R $ denotes its primitive, i.e. $F_q(x,s)=\int\limits_{0}^s f_q(x,t)dt$. \\
Additionally, we assume that $f_q(x,0)=0$ for all $x\in \Omega$.
\item[(F2)] There exists $\kappa_2>0$ such that $F_q(x,s)\geq \kappa_2$ for all $x\in \Omega$, $s\in \{-\epsilon_0, \epsilon_0\}$.
\item[(F3)] For every $s\in (-\epsilon_0, \epsilon_0)$ the functions $f_q(\cdot,s)$ and $F_q(\cdot, s)$ are $C^1$ on $\Omega$ and
there exists $\kappa_1>0$ such that
\begin{align*}
|\nabla_x F_q(x,s)|&\leq \kappa_1 F_q(x,s) \mbox{ for all } x \in \Omega, \ s \in (-\epsilon_0,\epsilon_0),\\
|\nabla_x f_q(x,s)|&\leq \kappa_1 |f_q(x,s)| \mbox{ for all } x \in \Omega, \ s \in (-\epsilon_0, \epsilon_0).
\end{align*}
\item[(F4)] For all $x\in \Omega$ we have $f_q(x,\cdot)\in C^1((-\epsilon_0,\epsilon_0)\setminus \{0\})$ and the function $(x,s) \mapsto g_q(x,s):=s\p_s f_q(x,s)$ is well-defined with $g_q\in L^{\infty}(\Omega \times (-\epsilon_0,\epsilon_0))$. 
\end{itemize}
As in \cite{SW17} we remark that the condition (F1) implies that the function $s\mapsto \frac{F_q(x,s)}{|s|^q}$ is non-increasing on $(0,\epsilon_0)$, while it is non-decreasing on $(-\epsilon_0,0)$. In particular, combined with the condition (F2), it provides a lower bound of the form
\begin{align*}
F_{q}(x,s) \geq \frac{\min\{F_q(x,\epsilon_0), F_q(x,-\epsilon_0)\}}{\epsilon_0^q}|s|^q \geq \frac{\kappa_2}{\epsilon_0^q} |s|^q.
\end{align*}
To simplify notation, we introduce the following abbreviations, which we will use frequently in the sequel:
\begin{align*}
\hat{f}_q(u)(x):=f_q(x,u(x)),\ \hat{F}_q(u)(x):=F_q(x,u(x)).
\end{align*}

In order to derive the strong unique continuation property and the unique continuation property from measurable sets, we in addition also make the following assumption:
\begin{itemize}
\item[(F5)] There exists $p\in (1,2)$ and $\kappa_3>0$ such that $|f_q(x,s)|\leq \kappa_3 |s|^{p-1}$ for $x\in \Omega$ and $s\in (-\epsilon_0, \epsilon_0)$.
\end{itemize}
In particular, by the definition of $F_q$ the condition (F5) also entails that
\begin{align}
\label{eq:F5_1}
|F_q(x,s)| \leq \kappa_3 |s|^{p} \mbox{ for all } x \in \Omega, \ s \in (-\epsilon_0,\epsilon_0).
\end{align}

\begin{example}
As in \cite{SW17} we remark that an example of an equation for which the conditions (A1)-(A3) and (F1)-(F5) are satisfied is for instance given by
\begin{align*}
\D u + \sum\limits_{j=1}^{m} c_j(x) |u|^{q_j-2}u = V u,
\end{align*}
where $q_j \in (1,2)$, $c_j, \nabla c_j \in L^{\infty}(\Omega)$ with $c_j>0$, $V \in L^{\infty}(\Omega)$ and $m\in \N$. In particular, the function $f_q(x,s)$ need not have a fixed power growth in $s$, but could for instance consist of a sum of different powers. 
\end{example}

Under these conditions, we then study a variable coefficient analogue of the model problem \eqref{eq:model} and prove that analogous unique continuation properties hold:

\begin{thm}
\label{thm:SMUCP_var}
Let $\Omega \subset \R^n$ be open, let $x_0 \in \Omega$ and let $V\in L^{\infty}(\Omega)$.
Let $u\in H^1_{loc}(\Omega)\cap L^{\infty}_{loc}(\Omega)$ be a solution of 
\begin{align}
\label{eq:model_1}
\p_i a^{ij}\p_j u + \hat{f}_q(u)= V u \mbox{ in } \Omega, 
\end{align}
where the conditions (A1)-(A3) and (F1)-(F4) are assumed to be valid.
Suppose further that one of the following conditions holds:
\begin{itemize}
\item[(a)] $q\in [1,2)$ and there exists a radius $r_0>0$ such that  $u$ vanishes on $B_{r_0}(x_0)\subset \Omega$ for some $x_0 \in \Omega$.
\item[(b)] $q\in(1,2)$, the condition (F5) is satisfied and $u$ vanishes of infinite order at $x_0\in \Omega$, i.e. for any $m\in \N$ we have $\lim\limits_{r\rightarrow 0} r^{-m} \int\limits_{B_r(x_0)} u^2 dx = 0 $.
\item[(c)] $q\in (1,2)$ and $u$ vanishes on a measurable set of positive measure, i.e. there exists $E \subset \Omega$ such that $|E|>0$ and $u|_{E}=0$, and the condition (F5) is satisfied.
\end{itemize}
Then, $u \equiv 0$ in $\Omega$.
\end{thm}

Again the argument is based on a Carleman inequality (c.f. Theorem \ref{prop:varmet}), which is explained in more detail in Section \ref{sec:var}. In order to deal with the Lipschitz coefficients of the metric, we use the ``geodesic normal coordinates" introduced by Aronszajn, Krzywicki and Szarski in \cite{AKS62}. This is technically more involved than the proof of Theorem \ref{thm:sub_lin_Carl}, but relies on the same ideas.

\subsection{Outline of the article}
The remainder of the article is organized as follows: In Section \ref{sec:prelim} we first recall some basic properties of the solutions to \eqref{eq:model} and to \eqref{eq:model_1}. Then, in Section \ref{sec:CarlI}, we prove the main Carleman estimate, i.e. Theorem \ref{thm:sub_lin_Carl}, in the model case. Based on this, we show how such a Carleman estimate implies the desired results of Theorems \ref{thm:SMUCP} and \ref{thm:SMUCP_var} in Section \ref{sec:proofs}. Finally, in Section \ref{sec:var} we then conclude our argument by also deducing the variable coefficient Carleman estimate of Theorem \ref{prop:varmet}.

\section{Preliminaries}
\label{sec:prelim}

In this section, we describe several auxiliary results, which will be used in the proofs of Theorems \ref{thm:SMUCP} and \ref{thm:SMUCP_var}. 

We begin by defining the notation of a weak solution to  \eqref{eq:model_1}:

\begin{defi}
\label{defi:weak_sol}
Assume that the conditions (A1)-(A3) and (F1)-(F4) hold true.
Let $u\in H^1_{loc}(\Omega)\cap L^{\infty}_{loc}(\Omega)$ and assume that $x \mapsto f(x,u(x))$ is Lebesgue measurable. Then $u$ is a weak solution to \eqref{eq:model_1} if for all $\xi \in H^{1}_{loc}(\Omega)\cap L^{\infty}_{loc}(\Omega)$
\begin{align*}
(a \nabla u, \nabla \xi)_{L^2(\Omega)} -(\hat{f}_q( u), \xi)_{L^2(\Omega)} = (Vu, \xi)_{L^2(\Omega)}.
\end{align*}
\end{defi}

Let us discuss the regularity of these solutions: If we only use the assumption that $f_q \in L^{\infty}_{loc}(\Omega \times \R)$, a bootstrap argument of elliptic regularity estimates directly implies that $u\in W^{2,p}_{loc}$ for all $p\in (1,\infty)$ and also $u\in C^{1,\alpha}_{loc}$ for all $\alpha \in (0,1)$. Due to our strengthening of the regularity conditions (c.f. conditions (F3) and (F4)), this could even be further bootstrapped. As it is not necessary in the sequel, we do not discuss this further. 
Assuming that for some $x_0\in \Omega$ we have $u(x_0)=0$, these regularity results imply that we may always assume that for all $x\in \Omega$ we have $u(x)\in (-\epsilon_0, \epsilon_0)$, where $\epsilon_0>0$ is the constant from the conditions (F1)-(F4). Indeed, if this were not the case, we could simply decrease the size of $\Omega$. In the sequel, we will always assume that this has already been carried out. \\

In the following arguments, we will often use that solutions to \eqref{eq:model_1} satisfy elliptic gradient estimates:

\begin{lem}[Caccioppoli]
\label{lem:Cacc}
Let (A1)-(A3) and (F1)-(F4) hold.
Let $u\in H^{1}_{loc}(B_4)\cap L^{\infty}_{loc}(B_4)$ be a solution to
\begin{align*}
\p_i a^{ij} \p_j u + \hat{f}_q(u)= V u \mbox{ in } B_4, 
\end{align*}
where $V \in L^{\infty}(B_4)$. Then for any $r\in (0,2)$ we have
\begin{align*}
\|\nabla u\|_{L^2(B_{r})} \leq C \left(\left\|\frac{|\hat{F}_q(u)|^{1/2}}{|u|^{1/2}}\right\|_{L^2(B_{2r})} + r^{-1}\|u\|_{L^2(B_{2r})} \right).
\end{align*}
\end{lem}

Here and in the sequel we have used the notation $B_{r}=B_{r}(0)$ in order to denote the ball of radius $r>0$ centered at zero.

\begin{proof}
The proof follows from the usual integration by parts identities. Indeed, let $\eta:B_{4}\rightarrow [0,\infty)$ be a cut-off function, which is equal to one in $B_{r}$ and which vanishes outside of $B_{2r}$ and which satisfies $|\nabla \eta|\leq \frac{C}{r}$, $|D^2 \eta|\leq \frac{C}{r^2}$. Then,
\begin{align*}
\lambda \int\limits_{B_{2r}}|\nabla (u\eta)|^2 dx 
&\leq \int\limits_{B_{2r}} a^{ij} \p_i(u\eta) \p_j(u \eta) dx
= \int\limits_{B_{2r}} a^{ij} \eta(\p_i u)\p_j(u\eta) dx + \int\limits_{B_{2r}} a^{ij} u (\p_i \eta) \p_j(u\eta)dx\\
& = \int\limits_{B_{2r}} a^{ij}\p_i u \p_j(u \eta^2) dx 
+ \int\limits_{B_{2r}} a^{ij} u^2 (\p_i \eta)( \p_j \eta) dx\\
&= \int\limits_{B_{2r}} \hat{f}_q(u) u \eta^2 dx + \int\limits_{B_{2r}} Vu^2 \eta^2 dx + \int\limits_{B_{2r}} u^2 a^{ij} (\p_i \eta)(\p_j \eta) dx\\
& \leq C_{\lambda, \Lambda,n}(\|V\|_{L^{\infty}}+1 + r^{-2})\|u\|_{L^2(B_{2r})}^2 + \left\| \eta \frac{|\hat{F}_q(u)|^{1/2}}{|u|^{1/2}} \right\|_{L^2(B_{2r})}^2.
\end{align*}
Here we used the (weak) equation as well as Hölder's inequality. Using that $\eta=1$ on $B_r$ then implies the desired estimate.
\end{proof}

Next, we show that the infinite order of vanishing can be equivalently defined by various different norms for solutions to \eqref{eq:model_1}.

\begin{lem}[Order of vanishing]
\label{lem:order_of_van}
Let the conditions (A1)-(A3) and (F1)-(F4) hold.
Let $u\in H^{1}_{loc}(B_4)\cap L^{\infty}_{loc}(B_4)$ be a solution to
\begin{align*}
\p_i a^{ij} \p_j u + \hat{f}_q(u)= V u \mbox{ in } B_4, 
\end{align*}
where $V \in L^{\infty}(B_4)$ and $q\in (1,2)$. 
Then the following are equivalent:
\begin{itemize}
\item[(i)] For all $m\in \N$ we have 
\begin{align}
\label{eq:L2_oov}
\lim\limits_{r\rightarrow 0} r^{-m} \int\limits_{B_{r}(0)} u^2 dx = 0.
\end{align}
\item[(ii)] For some $\ell >0$ and all $m\in \N$ we have 
\begin{align}
\label{eq:Lq_oov}
\lim\limits_{r\rightarrow 0} r^{-m} \int\limits_{B_{r}(0)} |u|^{\ell} dx = 0.
\end{align}
\end{itemize}
If in addition the condition (F5) is satisfied and if (i) holds, then also the function $\hat{f}_q(u)$ vanishes of infinite order, i.e. for all $m\in \N$ 
\begin{align*}
\lim\limits_{r\rightarrow 0} r^{-m} \int\limits_{B_{r}(0)} \hat{f}_q(u) dx = 0. 
\end{align*}
\end{lem}

\begin{proof}
Due to the regularity of solutions to \eqref{eq:model_1} and due to the assumption that $\ell>0$, the implication $\eqref{eq:L2_oov} \Rightarrow \eqref{eq:Lq_oov}$ follows from Hölder's inequality and the fact that $u\in L^{\infty}_{loc}$. The reverse implication follows from the assumption that $u \in L^{\infty}_{loc}$.

In the case that the condition (F5) is satisfied, we have
\begin{align*}
0\leq \lim\limits_{r\rightarrow 0} r^{-m} \int\limits_{B_{r}(0)} \hat{f}_q(u) dx \leq \kappa_3 \lim\limits_{r\rightarrow 0} r^{-m} \int\limits_{B_{r}(0)} |u|^{p-1} dx =0,
\end{align*}
where for the last inequality, we used the equivalence of (i) and (ii) and the fact that $p-1>0$.
\end{proof}

\section{The Carleman Estimate in the Model Set-Up}

\label{sec:CarlI}

In this section, we present the argument for Theorem \ref{thm:sub_lin_Carl}. The variable coefficient analogue will be proved in Section \ref{sec:var}, where we deal with the full problem (which also involves Lipschitz metrics).


\subsection{Proof of Theorem \ref{thm:sub_lin_Carl}}
The main idea leading to the Carleman estimate from Theorem \ref{thm:sub_lin_Carl} is to include the sublinear potential into the main operator instead of dealing with it perturbatively (as one would usually do for superlinear potentials).

\begin{proof}[Proof of Theorem \ref{thm:sub_lin_Carl}]
We separate the proof into several steps:\\

\emph{Step 1: Conjugation.}
We introduce conformal polar coordinates $x=e^{t}\theta$ with $(t,\theta) \in \R \times S^{n-1}$. In these the Laplacian reads
\begin{align*}
|x|^2\D = \p_t^2 + (n-2)\p_t + \D_{S^{n-1}}.
\end{align*}
Conjugating this with $e^{- \frac{n-2}{2}t}$ yields
\begin{align*}
e^{\frac{n-2}{2}t} (\p_t^2 + (n-2)\p_t + \D_{S^{n-1}}) e^{- \frac{n-2}{2}t} 
= \p_t^2 - \frac{(n-2)^2}{4} + \D_{S^{n-1}}.
\end{align*}
To achieve this, we consider the function $\tilde{v}(t,\theta) := e^{-\frac{n-2}{2}t}u(e^{t}\theta)$. The equation \eqref{eq:eq_main} then turns into
\begin{align}
\label{eq:eq1}
\left(\p_t^2 - \frac{(n-2)^2}{4} + \D_{S^{n-1}} + e^{2t}\frac{\tilde{f}_q(\tilde{u})}{\tilde{u}}\right) \tilde{v} = \tilde{g},
\end{align}
where $\tilde{g}(t,\theta) = e^{2t} g(e^t \theta)$, $\tilde{u}(t,\theta):=u(e^{t}\theta)$ and $\tilde{f}_q(\tilde{u})(t,\theta):=\hat{f}_q(\tilde{u})(e^t\theta)$. In order to prove the Carleman estimate from Theorem \ref{thm:sub_lin_Carl}, we argue by means of the usual conjugation argument and conjugate \eqref{eq:eq1} with the weight $e^{\tau \varphi}$, where $\varphi(t) = \psi(e^t)$. This yields the following symmetric and antisymmetric parts for the conjugated operator $L_{\varphi}:=S+A$:
\begin{align}
\label{eq:separate}
\begin{split}
S &= \p_t^2 + \D_{S^{n-1}} + \tau^2 (\varphi')^2 - \frac{(n-2)^2}{4} + h_q(\tilde{u}),\\
A &= -2\tau \varphi' \p_t - \tau \varphi''.
\end{split}
\end{align}
For ease of notation, we have abbreviated $h_q(\tilde{u}):= e^{2t}\frac{\tilde{f}_q(\tilde{u})}{\tilde{u}}$.
The Carleman estimate then follows from the expansion
\begin{align}
\label{eq:expand}
\|A v\|_{L^2}^2 + \|Sv\|_{L^2}^2 + \int\limits_{\R \times S^{n-1}}([S,A]v,v)dt d\theta = \|Lv\|_{L^2}^2,
\end{align}
where $v= e^{\tau \varphi} \tilde{v}$ and where we abbreviate $(\cdot, \cdot):=(\cdot,\cdot)_{L^2(\R \times S^{n-1})}$.
More precisely, the Carleman estimate follows, if we can prove lower bounds for the commutator $[S,A]$.
With respect to the usual commutator estimate for $L^2$ Carleman estimates, only the terms involving $h_q$ are new. Indeed, by choosing $\tau \geq \tau_0>1$ for some sufficiently large constant $\tau_0$ and recalling our choice of $\varphi$, the ``standard commutator term" $([\p_t^2 + \D_{S^{n-1}} + \tau^2 (\varphi')^2 - \frac{(n-2)^2}{4}, -2\tau \varphi' \p_t - \tau \varphi'']v,v)$ can be controlled as follows
\begin{align}
\label{eq:standard_Carl}
\begin{split}
&([\p_t^2 + \D_{S^{n-1}} + \tau^2 (\varphi')^2 - \frac{(n-2)^2}{4}, -2\tau \varphi' \p_t - \tau \varphi'']v,v)\\
&= 4 \tau^3 (\varphi'' (\varphi')^2 v,v) - \tau (\varphi'''' v, v)
+ 4 \tau (\varphi'' \p_t v, \p_t v)\\
&\geq 3\tau^{3}\||\varphi''|^{1/2} v\|_{L^2}^2 + 3 \tau \||\varphi''|^{1/2} \p_t v\|_{L^2}^2 .
\end{split}
\end{align}
Hence, in Step 2, we mainly consider the new, nonlinear contribution $([h_q(\tilde{u}),-2\tau \varphi' \p_t - \tau \varphi''] v, v)$. In Step 3, we then exploit the commutator estimate for the sublinear term together with the symmetric operator $S$ in order to upgrade the (radial) gradient estimate from \eqref{eq:standard_Carl} to a full gradient estimate (also controlling the spherical part of the gradient).\\

\emph{Step 2: The sublinear nonlinearity.}
We compute
\begin{align}
\label{eq:comm_0}
-2\tau [ h_q(\tilde{u}), \varphi' \p_t ]
= 2 \tau (q-2) \varphi' e^{2t} \sgn(\tilde{u})|\tilde{u}|^{q-3} \p_t \tilde{u} + 4\tau \varphi' h_q(\tilde{u}).
\end{align}
As a consequence, by the identity $v=e^{\tau \varphi} e^{\frac{2-n}{2}t} \tilde{u}$,
\begin{align}
\label{eq:comm}
\begin{split}
-2\tau (v,[ h_q(\tilde{u}), \varphi' \p_t ] v)
&= 2\tau (q-2) (e^{2t} \varphi' v^2 \sgn(\tilde{u})|\tilde{u}|^{q-3}, \p_t \tilde{u}  )
 + 4 \tau (e^{2t}\varphi' |\tilde{u}|^{q-2}v,v)\\
&= 2 \tau (q-2) ( l(t)  |\tilde{u}|^{q-2} \tilde{u}, \p_t \tilde{u}) 
+ 4 \tau ( e^{2t}\varphi' v, |\tilde{u}|^{q-2}v),
\end{split}
\end{align}
where all scalar products are those of the Hilbert space $L^2(\R \times S^{n-1})$, $l(t):= e^{(2-n)t}e^{2t}\varphi'(t)e^{2 \tau \varphi}$. We further study the first term on the right hand side of \eqref{eq:comm}: Noting that $|\tilde{u}|^{q-2}\tilde{u} \p_t \tilde{u}= \p_t \frac{1}{q}|\tilde{u}|^{q}$ leads to
\begin{align*}
( l(t) |\tilde{u}|^{q-2} \tilde{u}, \p_t \tilde{u} )
= - \frac{1}{q}( l'(t) |\tilde{u}|^{q-2} \tilde{u},\tilde{u} ).
\end{align*}
Inserting this back into \eqref{eq:comm} implies
\begin{align}
\label{eq:comm2}
\begin{split}
-2\tau (v,[ h_q(\tilde{u}), \varphi' \p_t ] v)
&= -\tau \frac{2(q-2)}{q}( l'(t) |\tilde{u}|^{q-2} \tilde{u},\tilde{u})
+ 4 \tau (e^{2t}\varphi' v |\tilde{u}|^{q-2},v)\\
& = -\tau \frac{2(q-2)}{q}\left( e^{2t} \left(\varphi''(t) + 2 \tau (\varphi'(t))^2 +\left(4-n\right)\varphi'(t) \right)|\tilde{u}|^{q-2} v, v \right)\\
& \quad + 4 \tau (e^{2t}\varphi' v |\tilde{u}|^{q-2},v).
\end{split}
\end{align}
Since $\varphi'' \geq 0$ and $q-2\leq 0$ the first two terms in  \eqref{eq:comm2} are positive. The last two terms are not necessarily signed, but by choosing $\tau \geq \tau_0>0$ sufficiently large and by recalling the explicit choice of our Carleman weight $\varphi$, they can be absorbed into the second contribution. In particular, combining the estimates \eqref{eq:standard_Carl} and \eqref{eq:comm2}, then leads to the estimate
\begin{align}
\label{eq:est_wo_grad}
\tau \left(\frac{2-q}{q} \right)^{\frac{1}{2}} \|e^{t} |\tilde{u}|^{\frac{q-2}{2}}v\|_{L^2} + \tau^{\frac{3}{2}} \||\varphi''|^{\frac{1}{2}}v\|_{L^2} + \tau^{\frac{1}{2}} \||\varphi''|^{\frac{1}{2}} \p_t v \|_{L^2} \leq C \|Lv\|_{L^2}.
\end{align}
Returning to Cartesian coordinates, this yields all the terms in the Carleman estimate, with the exception of the estimate for the spherical component of the gradient.\\

\emph{Step 3: Deriving the full gradient estimate.} Last but not least, we upgrade the gradient estimate from \eqref{eq:standard_Carl}, which only involves the radial derivatives to a full gradient estimate. To this end, we exploit the symmetric part $S$ of the operator. Indeed, testing the symmetric part with $\tau c_q \varphi'' v$ for a sufficiently small constant $c_q \in \left(0,\left(\frac{q-2}{q}\right)^{1/2}\right)$ and using \eqref{eq:est_wo_grad}, we infer
\begin{align*}
c_q \tau \||\varphi''|^{1/2} \nabla_{S^{n-1}}v\|_{L^2}^2
&\leq c_q \tau |(Sv, \varphi'' v)| + c_q \tau \||\varphi''|^{1/2} \p_t v\|_{L^2}^2 + c_q \tau^3 \||\varphi''|^{1/2} v\|_{L^2}^2
+ c_q \tau \|e^t |\tilde{u}|^{\frac{q-2}{2}}v\|_{L^2}^2\\
& \leq \frac{1}{2}\|S v\|_{L^2}^2 + C c_q \tau^2 \||\varphi''|^{1/2}v\|_{L^2}^2\\
& \quad + c_q \tau \||\varphi''|^{1/2} \p_t v\|_{L^2}^2
+ c_q \tau^3 \||\varphi''|^{1/2} v\|_{L^2}^2 + c_q \tau \|e^t|\tilde{u}|^{\frac{q-2}{2}}v\|_{L^2}^2\\
&\leq \frac{1}{2}\|S v\|_{L^2}^2 + ([S,A]v,v)
\leq \|L v\|_{L^2}^2.
\end{align*}
As a consequence, we may include the full gradient term into the Carleman estimate. This concludes the proof of Theorem \ref{thm:sub_lin_Carl}.
\end{proof}

\section{Proof of Theorems \ref{thm:SMUCP} and \ref{thm:SMUCP_var}}
\label{sec:proofs}

In this section we present the proof of Theorems \ref{thm:SMUCP} and \ref{thm:SMUCP_var} starting from the corresponding Carleman estimates (Theorems \ref{thm:sub_lin_Carl} and \ref{prop:varmet}). For the variable coefficient setting, the corresponding Carleman estimate will be proved in Section \ref{sec:Carl_var}.
By the regularity estimates and the discussion in Section \ref{sec:prelim}, we may assume that for $x\in \Omega$ we have $u(x) \in (-\epsilon_0, \epsilon_0)$. 

In the sequel, we first prove part (b) of Theorems \ref{thm:SMUCP} and \ref{thm:SMUCP_var}, which in particular also implies (a) in the case $q\in (1,2)$. Then we explain the modifications that allow us to prove the property (a) in the limiting case $q=1$. Last but not least, we explain the derivation of part (c) of the corresponding theorems.

\subsection{Proof of Theorems \ref{thm:SMUCP}(b) and \ref{thm:SMUCP_var}(b)}

The proof of the SUCP is a direct consequence of the Carleman estimate. Indeed, we apply it to a cut-off of $u$. Using the vanishing of infinite order, we are able to remove the cut-off around zero, if $q\in (1,2)$.

\begin{proof}[Proof of Theorems \ref{thm:SMUCP}(b) and \ref{thm:SMUCP_var}(b)]
Since $\Omega$ is open, translation and scaling allows us to assume that $B_{4} \subset \Omega$.

For $\epsilon \in (0,1)$ let $\eta_{\epsilon}: B_4 \rightarrow (0,\infty)$ be a cut-off function, which is supported in $B_{2}\setminus B_{\epsilon}$, which is equal to one in $B_1\setminus B_{2 \epsilon}$ and which satisfies the bounds
\begin{align}
\label{eq:eta}
\begin{split}
& |\nabla \eta_{\epsilon}(x)| \leq \frac{C}{\epsilon}, \ 
|D^2 \eta_{\epsilon}(x)| \leq \frac{C}{\epsilon^2} \mbox{ for all } x \in B_{2 \epsilon} \setminus B_{\epsilon},\\
& |\nabla \eta_{\epsilon}(x)| \leq C, \ 
|D^2 \eta_{\epsilon}(x)| \leq C \mbox{ for all } x \in B_{2 } \setminus B_{1},
\end{split}
\end{align}
where $C>0$ is independent of $\epsilon>0$. Then the function $v_{\epsilon}:= u \eta_{\epsilon}$ satisfies
\begin{align}
\label{eq:cut_off_eq}
\p_i a^{ij} \p_j v_{\epsilon} + \hat{f}_q(v_{\epsilon}) = V v_{\epsilon} +(\hat{f}_q(v_{\epsilon})-\eta_{\epsilon} \hat{f}_q(u)) + 2 a^{ij}\p_i \eta_{\epsilon} \p_j u
+ u \p_i a^{ij}\p_j \eta_{\epsilon} \mbox{ in } B_4.
\end{align}
We apply the Carleman estimate from Theorem \ref{thm:sub_lin_Carl} to $v_{\epsilon}$, which leads to
\begin{align}
\label{eq:Carl_appl}
\begin{split}
&\tau^{3/2}\|e^{\tau \phi} (1+\ln^2(|x|))^{-\frac{1}{2}} v_{\epsilon}\|_{L^2(B_2)}
+ \tau^{1/2} \|e^{\tau \phi} (1+\ln^2(|x|))^{-\frac{1}{2}}|x| \nabla v_{\epsilon}\|_{L^2(B_2)}\\
& +\tau \left( \frac{q-2}{q} \right)^{\frac{1}{2}} \|e^{\tau \phi} |x||v_{\epsilon}|^{\frac{q}{2}}\|_{L^2(B_2)} \\
&\leq C \left( \|e^{\tau \phi} |x|^2 V v_{\epsilon}\|_{L^2(B_2)} + \|e^{\tau \phi}|x|^2(\hat{f}_q(v_{\epsilon})-\eta_{\epsilon}\hat{f}_{q}(u))\|_{L^2(B_2)} + 2 \Lambda \|e^{\tau \phi} |x|^2 |\nabla \eta_{\epsilon}|| \nabla u|\|_{L^2(B_2)} \right.\\
& \quad \left. + \|e^{\tau \phi}|x|^2 u |D^2 \eta_{\epsilon}|\|_{L^2(B_2)}  \right).
\end{split}
\end{align}
We seek to pass to the limit $\epsilon \rightarrow 0$. Since none of the constants in the estimate \eqref{eq:Carl_appl} depends on $\epsilon>0$ and using the bounds in \eqref{eq:eta}, this can be achieved by invoking the infinite order of vanishing of $u$. Indeed, this directly allows us to pass to the limit $\epsilon \rightarrow 0$ in all $L^2$ terms of $v_{\epsilon}$ or $u$. In order to deal with the gradient terms and the nonlinearity on the left hand side, we apply Lemmas \ref{lem:Cacc} and \ref{lem:order_of_van}. For the nonlinearity on the right hand side, we use the condition (F5) in combination the second part of Lemma \ref{lem:order_of_van}.

Setting $v_0:=\eta_0 u$ (where $\eta_0$ is the pointwise limit of $\eta_{\epsilon}$; in particular $\eta_0=1$ in $B_1$), then implies
\begin{align}
\label{eq:Carl_appl_1}
\begin{split}
&\tau^{3/2}\|e^{\tau \phi} (1+\ln^2(|x|))^{-\frac{1}{2}}v_0\|_{L^2(B_2)}
 + \tau^{1/2} \|e^{\tau \phi} (1+\ln^2(|x|))^{-\frac{1}{2}}|x| \nabla v_0\|_{L^2(B_2)}\\
& + \tau \left( \frac{q-2}{q} \right)^{\frac{1}{2}} \|e^{\tau \phi} |x||v_0|^{\frac{q}{2}}\|_{L^2(B_2)}\\
&\leq C \left( \|e^{\tau \phi} |x|^2 V v_0\|_{L^2(B_2)} + \|e^{\tau \phi}|x|^2(\hat{f}_q(v_0)-\hat{f}_{q}(u))\|_{L^2(B_2)} + 2 \|e^{\tau \phi} |x|^2 |\nabla \eta_0|| \nabla u|\|_{L^2(B_2)} \right.\\
& \quad \left. + \|e^{\tau \phi}|x|^2 u |D^2 \eta_{0}|\|_{L^2(B_2)}  \right).
\end{split}
\end{align}
By virtue of the $L^{\infty}$ boundedness of $V$, we may further absorb the first term on the right hand side of \eqref{eq:Carl_appl_1} into the left hand side of \eqref{eq:Carl_appl_1} if $\tau$ is chosen such that $\tau \geq \tau_0(\|V\|_{L^{\infty}})$. Using that $\eta_0=1$ in $B_1$, which in particular entails that
\begin{align*}
\supp(|\nabla \eta_0||\nabla u|), \supp(|D^2 \eta_0||u|), \supp(|\hat{f}_q(v_0)-\hat{f}_q(u)|) \subset \overline{B_2 \setminus B_1},
\end{align*}
we then further estimate
\begin{align*}
\begin{split}
&e^{\tau \psi(1/2)} \left(\tau^{3/2}\| (1+\ln^2(|x|))^{-\frac{1}{2}} v_0 \|_{L^2(B_{1/2})}
 + \tau^{1/2} \|(1+\ln^2(|x|))^{-\frac{1}{2}} |x| \nabla v_0\|_{L^2(B_{1/2})} \right.\\
 & \left. + \tau \left( \frac{q-2}{q} \right)^{\frac{1}{2}} \||x| |v_0|^{\frac{q}{2}}\|_{L^2(B_{1/2})} \right)\\
&\leq C e^{\tau \psi(1)} \left( \||x|^2(\hat{f}_q(v_0)-\hat{f}_{q}(u))\|_{L^2(B_2\setminus B_1)} + 2  \| |x|^2 |\nabla \eta_0||\nabla u|\|_{L^2(B_2\setminus B_1)} \right.\\
& \quad \left. +  \||x|^2 u |D^2 \eta_{0}|\|_{L^2(B_2\setminus B_1)}  \right).
\end{split}
\end{align*}
Dividing by $e^{\tau \psi(1/2)}$ and passing to the limit $\tau \rightarrow \infty$ (and recalling the a priori estimates for $u$) then implies that $u\equiv 0$ in $B_{1/2}$. Iterating this argument yields that $u\equiv 0$ in $\Omega$.
\end{proof}

We remark that this proof simultaneously deals with the situation of Theorems \ref{thm:SMUCP} and \ref{thm:SMUCP_var}.

\subsection{Proof of Theorems \ref{thm:SMUCP}(a) and \ref{thm:SMUCP_var}(a)}

Without loss of generality, we may assume that $u\equiv 0$ in $B_{r_0}(0)$ for some $r_0 \in (0,1/4)$.
The proofs of Theorems \ref{thm:SMUCP}(a) and \ref{thm:SMUCP_var}(a) then proceed analogously to the one, which was explained in the previous subsection. However, as $u\equiv 0$ in $B_{r_0}(0)$, we do not need to use a cut-off function close to zero, but can directly consider a bump function: More precisely, we could consider a cut-off $\eta_{0}:B_4 \rightarrow (0,\infty)$ which is supported in $B_2$, is equal to one in $B_{1/2}$ and satisfies the bounds
\begin{align*}
|\nabla \eta_{0}|\leq C, \ |D^2 \eta_{0}| \leq C \mbox{ for all } x \in B_2 \setminus B_1.
\end{align*}
As a consequence, inserting $v_0:= u \eta_0$ into the Carleman estimates from Theorems \ref{thm:sub_lin_Carl} and \ref{prop:varmet}, we directly infer the estimate \eqref{eq:Carl_appl_1}, from which we conclude as in the previous section.

\subsection{Proof of Theorems \ref{thm:SMUCP}(c) and \ref{thm:SMUCP_var}(c)}
We reduce the statement of Theorems \ref{thm:SMUCP}(c) and \ref{thm:SMUCP_var}(c) to that of \ref{thm:SMUCP}(a) and \ref{thm:SMUCP_var}(a) by proving a suitable growth estimate.

\begin{proof}[Proof of Theorems \ref{thm:SMUCP}(c) and \ref{thm:SMUCP_var}(c)]
Assume that there exists a measurable set $E \subset \Omega$ such that
$|E|>0$ and $u\equiv 0$ on $E$. By translation, without loss of generality we may assume that $0\in E$ and that $0$ is a point of density one of $E$. In particular, for any $\epsilon>0$ there exists a radius $r_{\epsilon}>0$ such that
\begin{align}
\label{eq:dense}
\frac{|E \cap B_{r}|}{|E|} \leq \epsilon \mbox{ for all } r \in (0,r_{\epsilon}).
\end{align}
Thus, for $r\in (0,\min\{r_{\epsilon},1/2\})$ we obtain for some constant $C>1$, which depends on $\|V\|_{L^{\infty}(B_2)}, \lambda, \Lambda, \Lambda_0$ and $n$ and which may change from line to line,
\begin{align}
\label{eq:growth}
\begin{split}
\|u\|_{L^2(B_{r})}
&= \|u\|_{L^2(B_{r}\cap E)}
\leq |E \cap B_{r}|^{1/n} \|u\|_{L^{2*}(B_r \cap E)}
\leq |E \cap B_{r}|^{1/n}\|\nabla u\|_{L^2(B_r)}\\
&\leq C|E\cap B_{r}|^{1/n}\left( r^{-1}\|u\|_{L^2(B_{2r})} + \left\|\frac{|F_q(u)|^{1/2}}{|u|^{1/2}}\right\|_{L^2(B_{2r})} \right)\\
&\stackrel{(F5)}{\leq} C|E\cap B_{r}|^{1/n}\left( r^{-1}\|u\|_{L^2(B_{2r})} + \left\||u|^{\frac{p-1}{2}}\right\|_{L^2(B_{2r})} \right)\\
&\leq C \epsilon^{\frac{1}{n}} r (r^{-1} + r^{\gamma(p,n)})\|u\|_{L^2(B_{2r})}\\
&\leq C \epsilon^{1/n} \|u\|_{L^2(B_{2r})}
\end{split}
\end{align}
for some $\gamma(p,n)>0$ (which is obtained by an application of Hölder's inequality).
Here we have used the vanishing of $u$ on $E$, the condition (F5) in combination with Hölder's and Sobolev's inequalities, the fact that $0<r\leq 1$ and the density estimate \eqref{eq:dense}. Next, we fix $m\in \N$ and choose $\epsilon>0$ such that
\begin{align*}
C \epsilon^{1/n} \leq \frac{1}{2^m}.
\end{align*}
This, then implies the growth estimate
\begin{align*}
\|u\|_{L^2(B_{r})} \leq 2^{-m}\|u\|_{L^2(B_{2r})}
\end{align*}
This can be iterated as long as $2^{k} r \leq r_m:=r_{\frac{1}{2^m}}$ (with $r_{\frac{1}{2^m}}$ denoting the corresponding radius in \eqref{eq:dense}). In particular, it implies that for $r\in (2^{-k-1},2^{-k})$ we have
\begin{align*}
\|u\|_{L^2(B_{r})} \leq 2^{-km}\|u\|_{L^2(B_{r_{m}})}.
\end{align*}
As we can argue in the same way for any $m\in \N$, we infer the infinite order of vanishing of $u$ at $x_0=0$. The strong unique continuation property from Theorem \ref{thm:SMUCP}(a) then implies that $u\equiv 0$ in $\Omega$.
\end{proof}

\begin{rmk}
\label{rmk:q0}
We emphasize that from a technical point of view the ``only" obstruction in the above arguments preventing us from also deriving the SUCP and the MUCP for the case $q=1$ consists of justifying the support assumption, which we used above, for the strong $L^2$ limit 
\begin{align*}
\lim\limits_{\epsilon \rightarrow 0}e^{\tau \phi} |x|^2 (\hat{f}_q(v_{\epsilon}) - \eta_{\epsilon} f_q(u)).
\end{align*}
The second main technical point, in which we used $q>1$, i.e. the estimate \eqref{eq:growth}, could have easily been modified to work in the case $q=1$ by relying on the estimate
\begin{align*}
\|\nabla u\|_{L^2(B_r)} 
\leq C(\|V\|_{L^{\infty}}) \left( r^{-1}  \|u\|_{L^2(B_{2r})}+ \|\hat{F}_q(u)\|_{L^1(B_{2r})}\right), \ r \in (0,2),
\end{align*}
instead of invoking Lemma \ref{lem:Cacc} in the proof of \eqref{eq:growth}.
\end{rmk}

\section{The Case of More General Nonlinearities and Lipschitz Metrics}
\label{sec:var}

In this section we consider the setting described in Section \ref{sec:results_var} which involves equations with more general nonlinearities $\hat{f}_q(u)$ and with Lipschitz metrics. Throughout this section, we assume that the conditions (A1)-(A3) and (F1)-(F4) hold. Similarly as the proof of Theorem \ref{thm:SMUCP}, the argument for Theorem \ref{thm:SMUCP_var} is crucially based on a Carleman estimate:

\begin{thm}[Variable coefficient Carleman estimate]
\label{prop:varmet}
Suppose that the conditions (A1)-(A3) and (F1)-(F4) hold.
Let $\phi(x)=\psi(|x|)$ with
\begin{align*}
\phi(r)= - \ln(r) + \frac{1}{10}\left( \ln(r)\arctan(\ln(r)) - \frac{1}{2} \ln(1+\ln^2(r)) \right).
\end{align*}
Assume that $u\in H^{1}_{loc}(\R^{n})\cap L^{\infty}_{loc}(\R^n)$ with $\supp{(u)} \subset \overline{B_{r_0}\setminus B_{\epsilon}}$, where $0<\epsilon\ll r_0\ll 1$, satisfies 
\begin{align*}
\p_i a^{ij} \p_j u + \hat{f}_q(u) & = g \mbox{ in } \R^{n}.
\end{align*}
Then there exists $\tau_0>0$ (depending on $n,q,\lambda,\Lambda, \Lambda_0,\kappa_1,\kappa_2$) such that for $\tau\geq \tau_{0}$ we have
\begin{equation}
\label{eq:vCarl}
\begin{split}
 \tau^{\frac{3}{2}} \left\|  e^{\tau \phi} (1+\ln(|x|)^2)^{-\frac{1}{2}}  u \right\|_{L^2(\R^{n+1}_+)} +  \tau^{\frac{1}{2}}\left\| e^{\tau \phi} |x| (1+\ln(|x|)^2)^{-\frac{1}{2}}  \nabla u \right\|_{L^2(\R^{n+1}_+)}\\
+  \tau \left\|  e^{\tau \phi}|x| |u|^{\frac{q}{2}} \right\|_{L^2(\R^{n+1}_+)}
\leq C(q,\lambda,\Lambda,\Lambda_0,n,\kappa_1,\kappa_2)  \ \left\|e^{\tau \phi} |x|^2  g \right\|_{L^2(\R^{n+1}_+)}.
\end{split}
\end{equation}
\end{thm}

\begin{rmk}
\begin{itemize}
\item There are two main restrictions, which determine the size of the radius $r_0=r_0(q,n,\lambda, \Lambda,\Lambda_0,\kappa_1,\kappa_2)>0$ in the theorem: Firstly, we chose it so small that we may pass to suitable ``geodesic normal coordinates" in it. Secondly, we possibly impose even further restrictions on its size by requiring it to be sufficiently small in order to absorb some of the error terms, which arise in the proof of the Carleman estimate, into the leading order contributions of the Carleman estimate. This yields a dependence of $r_0$ on $n$ and $q$.
The smallness of $r_0$ is no restriction, since the UCP is a local property of an equation.
\item The proof of Theorem \ref{prop:varmet} illustrates that there are no additional difficulties in proving the Carleman estimate if additional lower order contributions are included in (\ref{eq:model_1}) as long as the coefficients remain bounded. For instance, it would have been possible to include bounded gradient potentials in our discussion.
\end{itemize}
\end{rmk}

In order to prove this low regularity, variable coefficient Carleman estimate, we use the coordinates introduced by Aronszajn, Krywicki and Szarski in \cite{AKS62}, who had introduced a ``replacement" of ``geodesic normal coordinates" in the presence of Lipschitz continuous metrics.
We recall this briefly in Section \ref{sec:AKS62}. Based on these ideas we introduce the corresponding ``geodesic polar coordinates" and carry out a similar conjugation argument as in the proof of the Carleman inequality in the model case (c.f. Section \ref{sec:Carl_var}). As explained in Section \ref{sec:proofs}, the proof of Theorem \ref{thm:SMUCP_var} then follows along the same lines as in the model situation.

\subsection{The coordinates of Aronszajn, Krzywicki and Szarski}
\label{sec:AKS62}

A priori the introduction of suitable geodesic coordinates poses difficulties in the case of Lipschitz metrics since the ODE system describing the geodesics does not posses well-defined, sufficiently regular solutions. Hence, we pursue a slightly different strategy following the ideas of Aronszajn, Krzywicki and Szarski \cite{AKS62}, who had found a way of introducing suitable ``geodesic normal coordinates" in a slightly different way also in the presence of Lipschitz continuous metrics. With these coordinates at hand, we then carry out a similar conjugation procedure as in the model setting in Section \ref{sec:model} from above. An alternative approach of dealing with the Lipschitz metrics as perturbations of constant coefficient metrics would also have been possible, c.f. \cite{KT01}. \\

We seek to present the proof of the variable coefficient Carleman estimate from Theorem \ref{prop:varmet} in a way which on the one hand avoids lengthy calculations and which on the other hand follows the arguments from Section \ref{sec:CarlI} as closely as possible. To this end, in the sequel we briefly recall a convenient change of coordinates due to Aronszajn, Krzywicki and Szarski \cite{AKS62}. Starting from a metric tensor $a_{ij}$ in a neighbourhood of the origin $B_{r_0}\subset \R^{n}$ which satisfies the conditions (A1)-(A3) from above, the authors of \cite{AKS62} introduce the following ``radial'' coordinate and a modified metric:
\begin{align}
& r=r(x):=(a_{ij}(0)x^i x^j)^{\frac{1}{2}},\label{eq:newmetric1}\\
& \tilde{a}_{ij}(x) := a_{ij}(x)\Psi(x), \label{eq:newmetric2}
\end{align}
where  
\begin{equation}\label{eq:newmetric3}
\Psi(x)=a^{k l}(x)\frac{\p r}{\p x^{k}} \frac{\p r}{\p x^{l}} \quad \text{for }x\neq 0,\quad \Psi(0)=1.
\end{equation}
and $(a^{kl})=(a_{kl})^{-1}$ is the inverse matrix.
Note that it is immediate from the uniform ellipticity that 
$$\frac{\lambda}{\Lambda}\leq \Psi(x)\leq \frac{\Lambda}{\lambda},$$
and that by definition $\Psi$ is Lipschitz continuous.

With these auxiliary quantities at hand, Aronszajn, Krywicki and Szarski construct the following replacement of geodesic polar coordinates:

\begin{prop}[\cite{AKS62}, Sections III, IV] 
\label{prop:AKS62}
In the ellipsoid 
$$\tilde{B}_{\tilde{r}_0}:=\{x\in \R^{n}:r(x)<\tilde{r}_0\}\subset B_{r_0} \quad \tilde{r}_0=r_0\sqrt{\lambda}$$ the following properties hold:
\item[(i)] $\tilde{a}_{ij}$ is uniformly elliptic with $\tilde{\lambda}=\lambda^2/\Lambda$, $\tilde{\Lambda}=\Lambda^2/\lambda$.
\item[(ii)] $\tilde{a}_{ij}$ is Lipschitz with Lipschitz constant
$\tilde{\Lambda}_0$ depending on $\Lambda_0$, $\lambda$ and $\Lambda$. 
\item[(iii) (Polar coordinates)] Let $\Sigma:=\partial\tilde{B}_{\tilde{r}_0}$. Then one can parametrize $\tilde{B}_{\tilde{x}_0}\setminus \{0\}$ by $r$ and $\theta$, with $r=r(x)$ defined in \eqref{eq:newmetric1} and $\theta=(\theta^1, \cdots,\theta^n)$ being a choice of local coordinates of $\Sigma$. In these coordinates, the metric turns into
\begin{align*}
\tilde{a}_{ij} dx^i dx^j=dr^2+r^2b_{k l}d \theta^k d \theta^{l} \text{ with } b_{k l}(r,\theta)=\frac{1}{r^2}\tilde{a}_{ij}\frac{\partial x^i}{\partial \theta ^k}\frac{\partial x^j}{\partial \theta^l}.
\end{align*}
\item[(iv)] There exists a constant $M=M(\lambda, \Lambda, \Lambda_0)$ such that for any tangent vector $\sigma\in T_{\theta}(\Sigma)$,
\begin{equation}\label{eq:polar_b}
\left|\frac{\partial b_{k l}(r,\theta)}{\partial r}\sigma^k\sigma^l\right|\leq M |b_{k l}(r,\theta)\sigma^k\sigma^l|.
\end{equation}
In particular, if we let $b:=\det(b_{k l})$, then \eqref{eq:polar_b} implies that
\begin{equation}\label{eq:logpolar_b}
\left|\frac{\partial(\ln(\sqrt{b})}{\partial r}\right|\leq \frac{nM}{2}.
\end{equation}
\end{prop}

In other words, the existence of the coordinates $(r,\theta)$, which is one of the central insights of the paper of Aronszajn, Krywicki and Szarski \cite{AKS62}, permits us to pass to ``geodesic polar coordinates" without explicitly making use of the system of ODEs defining the exponential map -- which, due to the low regularity of the metric, would not necessarily yield the desired choice of coordinates.

\subsection{Proof of Theorem \ref{prop:varmet}}
\label{sec:Carl_var}
With the conformal polar coordinates of \cite{AKS62} at hand, we discuss the proof of the Carleman estimate from Theorem \ref{prop:varmet}. In order to use these coordinates efficiently and to switch to the associated conformal polar coordinates, we rewrite our equation as a Laplace-Beltrami operator on the underlying manifold. This has the advantage that changes of coordinates can be easily computed.

\begin{lem}
\label{lem:geo}
Let $a^{ij}$ satisfy the assumptions (A1)-(A3) and let $\tilde{a}_{ij}$ be as \eqref{eq:newmetric2} in Section \ref{sec:AKS62}. Then the following are equivalent:
\begin{itemize}
\item[(i)] $u$ is a solution to
\begin{align*}
\p_i a^{ij} \p_j u + \hat{f}_q(u) = g \mbox{ in } \R^n.
\end{align*}
\item[(ii)] $u$ is a solution to
\begin{align}
\label{eq:LB}
\D_{\tilde{a}^{ij}} u + \frac{1}{\Psi} \hat{f}_q(u) 
= \frac{g}{\Psi} + \frac{1}{2 \Psi \tilde{a}} a^{ij} (\p_{x_i} \tilde{a}) \p_{x_j}u - a^{ij}\frac{\p_{x_i}\Psi}{\Psi^2} \p_{x_j}u=: \frac{g}{\Psi} + R=:h,
\end{align}
where $\D_{\tilde{a}^{ij}}$ denotes the Laplace-Beltrami operator with respect to the metric $\tilde{a}^{ij}$, $\tilde{a}:=\det(\tilde{a}_{ij})$ and $\Psi$ denotes the function from \eqref{eq:newmetric3} in Section \ref{sec:AKS62}.
\end{itemize}
\end{lem}

We omit the proof of this equivalence, as it follows from a direct calculation. Instead, we turn to the proof of Theorem \ref{prop:varmet}, for which we will rely on the geometric formulation of the Carleman estimate.

\begin{proof}[Proof of Theorem \ref{prop:varmet}]
\emph{Step 1: Choice of coordinates. }
Relying on Lemma \ref{lem:geo}, we prove a Carleman estimate for the operator $L u= \D_{\tilde{a}^{ij}} u + \frac{1}{\Psi} \hat{f}_q(u)$. The terms on the right hand side in Lemma \ref{lem:geo} (ii) will be treated as error contributions and will eventually be absorbed into the left hand side of the Carleman inequality (c.f. Step 5 below). 

Consider the geodesic polar coordinates $(r,\theta)\in (0,1)\times \Sigma$ from Proposition \ref{prop:AKS62} (iii), where $r(x)=|x|$ (since $a_{ij}(0)=\delta_{ij}$ by the normalization assumption (A3)) and $\theta$ are suitable coordinates of $\Sigma$.
By definition of the coordinates from \cite{AKS62}, we infer
\begin{align*}
\Delta_{\tilde{a}_{ij}}=&\frac{1}{r^n\sqrt{b}}\dr(r^{n}\sqrt{b}\dr) +\frac{1}{r^2}\Delta_{\Sigma}, \quad \sqrt{\tilde{a}}dx=r^n\sqrt{b}\ drd\theta,
\end{align*}
where 
$$\Delta_{\Sigma}=\frac{1}{\sqrt{b}}\p_{\theta_k} \left(b^{k l}\sqrt{b}\ \p_{\theta_{l}}\right), \quad b=\det(b_{k l}), \quad (b^{k l})=(b_{k l})^{-1}.$$ 

Next we carry out a change into conformal coordinates, i.e. $x=e^{t}\theta$, which in particular yields $\dr = e^{-t}\dt$. This resulting Laplace-Beltrami operator then reads
\begin{align*}
\D_{\tilde{a}_{ij}}=e^{-2t}\left[  \frac{1}{e^{(n-2)t}\sqrt{b}}\partial_t\left(e^{(n-2)t}\sqrt{b}\partial_t\right) +  \D_{\Sigma} \right].
\end{align*}

We conjugate the operator $\D_{\tilde{a}^{ij}}$ with the weight $e^{-\frac{n-2}{2}t}$ which leads to the representation
\begin{align*}
e^{\frac{n+2}{2}t}\D_{\tilde{a}_{ij}}e^{-\frac{n-2}{2}t}=\frac{1}{\sqrt{b}}\partial_t(\sqrt{b}\partial_t)-\left(\frac{n-2}{2}\right)^2 +  \D_{\Sigma}.
\end{align*} 
Hence, our equation \eqref{eq:LB} becomes
\begin{align}
\label{eq:eq_aux}
\left( \frac{1}{\sqrt{b}}\p_t(\sqrt{b}\p_t) - \left( \frac{n-2}{2} \right)^2 + \Delta_{\Sigma} + e^{2t} \frac{\hat{f}_q(\tilde{u})}{\Psi \tilde{u}} \right)\tilde{v} = \tilde{h},
\end{align}
where
\begin{align*}
\tilde{v}(t,\theta) = e^{- \frac{n-2}{2}t} u(e^{t}\theta), \
\tilde{h}(t,\theta) = e^{\frac{n+2}{2}t}h(e^t \theta), \
\tilde{u}(t,\theta) = u(e^t \theta).
\end{align*}

We conjugate \eqref{eq:eq_aux} with the weight $e^{\tau \varphi}$, where $\varphi(t) = \psi(e^t)$. The correspondingly conjugated operator turns into $L_{\varphi}=S+A$,
where
\begin{align}
\label{eq:sym_antisym}
\begin{split}
S & = \frac{1}{\sqrt{b}}\partial_t(\sqrt{b}\partial_t)+ \tau^2(\varphi')^2 - 
\left(\frac{n-2}{2}\right)^2 + \D_{\Sigma} +h_q(\tilde{u}),\\
A & = -2\tau \varphi'\dt - \tau \varphi''-\tau\varphi'\partial_t \ln(\sqrt{b}),
\end{split}
\end{align}
where $h_q(\tilde{u})= e^{2t} \frac{\hat{f}_q(\tilde{u})}{\Psi\tilde{u}}$.
We seek to derive the desired Carleman estimate by expanding the operator $L_{\varphi}$.\\

\emph{Step 2: Expansion of the operator $L_{\varphi}$.}
To estimate $\|L_\varphi u\|_{L^2}$, we use the splitting from \eqref{eq:sym_antisym} and expand the operator $L_{\varphi}$. Due to the low regularity of the metric, we do not directly phrase this as a commutator estimate, but morally it reduces to this.  

Due to the $t$-dependence of the volume element, we have an extra term $\tau\varphi'\partial_t \ln(\sqrt{b})$ in the antisymmetric part, whose $t$-derivative is not controlled. Thus, we treat this contribution as an error term, i.e. we split 
$$L_\varphi=S+\tilde{A}+E$$
where
\begin{align*}
\tilde{A}=-2\tau \varphi'\dt - \tau \varphi'', \quad E=-\tau\varphi'\partial_t \ln(\sqrt{b}).
\end{align*}
By virtue of the triangle inequality 
\begin{equation}\label{eq:tri}
\|(S+\tilde{A})v\|_{L^2_{\vol}}\leq \|L_\varphi v\|_{L^2_{\vol}}+\|Ev\|_{L^2_{\vol}}.
\end{equation}  
Here we have set $v = e^{\tau \varphi} \tilde{v}$ and, for simplicity of notation, we have abbreviated
\begin{align*}
\|\cdot \|_{L^2_{\vol}}= \|\cdot\|_{L^2_{\vol}(\R \times \Sigma)}=\|\cdot\|_{L^2(\R \times \Sigma, \sqrt{b}d \theta dt)}.
\end{align*}
The corresponding scalar product will be denoted by $(\cdot, \cdot)_{L^2_{\vol}}$.

We first notice that
\begin{align}
\label{eq:split11}
\|(S+\tilde{A})v\|^2_{L^2_{\vol}}=\|Su\|^2_{L^2_{\vol}} +\|\tilde{A}v\|^2_{L^2_{\vol}}+ 2(Su, \tilde{A}v)_{L^2_{\vol}},
\end{align}
We now estimate the contributions in $(S v, \tilde{A}v)_{L^2_{\vol}}$, which we split into three parts, which we consider separately:
\begin{align}
\label{eq:mixed}
\begin{split}
(S v, \tilde{A} v)_{L^2_{\vol}}
&= -2 \tau (\p_t (\sqrt{b} \p_t ) v + \tau^2 \sqrt{b}(\varphi')^2 v + \D_{\Sigma}' v - \sqrt{b}\frac{(n-2)^2}{4} v,  \varphi' \p_t v)_{L^2}\\
& \quad - \tau (\p_t (\sqrt{b} \p_t ) v + \tau^2 \sqrt{b}(\varphi')^2 v + \D_{\Sigma}' v - \sqrt{b}\frac{(n-2)^2}{4} v, \varphi'' v)_{L^2}\\
& \quad + ([h_q(\tilde{u})\sqrt{b}, -2\tau \varphi' \p_t - \tau \varphi''] v, v)_{L^2}.
\end{split}
\end{align}
Here we now consider the standard scalar product, i.e., $(\cdot,\cdot)_{L^2}:=(\cdot,\cdot)_{L^2(\R \times \Sigma)}$, and we have set $\D_{\Sigma'} := \sqrt{b} \D_{\Sigma}$.
We begin by discussing the first two contributions in \eqref{eq:mixed}, which do not involve the sublinear part of the problem. Hence, the main difficulty with these is to deal with the low regularity of the metric. To this end we compute,
\begin{align}
\label{eq:first_term}
\begin{split}
&-2 \tau \left(\p_t (\sqrt{b} \p_t ) v + \tau^2 \sqrt{b}(\varphi')^2 v + \D_{\Sigma}' v - \frac{(n-2)^2}{4}\sqrt{b} v,  \varphi' \p_t v\right)_{L^2}\\
&= -2\tau \left(\sqrt{b} \p_t^2 v + \tau^2 \sqrt{b}(\varphi')^3 v + \D_{\Sigma}' v- \frac{(n-2)^2}{4}\sqrt{b} v,  \varphi' \p_t v\right)_{L^2}
- \tau \left(\frac{b'}{\sqrt{b}}\p_t v, \varphi' \p_t v \right)_{L^2}\\
& = -\tau (\varphi' \sqrt{b}, \p_t(\p_t v)^2)_{L^2}
- \tau^3 (\sqrt{b} (\varphi')^2, \p_t (v^2))_{L^2}
- \tau (\frac{b'}{\sqrt{b}} \p_t v, \varphi' \p_t v)_{L^2}\\
& \quad - 2\tau (\D_{\Sigma}' v, \varphi' \p_t v)_{L^2} +\tau \frac{(n-2)^2}{4}(\varphi' \sqrt{b}, \p_t(v^2))_{L^2}\\
& = \tau ((\p_t v)\p_t (\varphi' \sqrt{b}), \p_t v)_{L^2}
+ \tau^3 (v \p_t (\sqrt{b} (\varphi')^3), v)_{L^2}
-\tau (\frac{b'}{\sqrt{b}} \p_t v, \varphi' \p_t v)_{L^2}
\\
& \quad + 2\tau (\nabla_{\theta}v, b \sqrt{b} \varphi' \p_t \nabla_{\theta} v)_{L^2}- \tau \frac{(n-2)^2}{4}(\varphi'' \sqrt{b} v, v)_{L^2} - \frac{\tau}{2} \frac{(n-2)^2}{4}(\varphi' \frac{b'}{\sqrt{b}} v, v)_{L^2}\\
& = \tau (\sqrt{b} \varphi'' \p_t v, \p_t v)_{L^2}
+ \frac{\tau}{2}(\varphi' \frac{b'}{\sqrt{b}}\p_t v, \p_t v)_{L^2}
+ 3 \tau^3 (\sqrt{b} \varphi'' \varphi' v, \varphi' v)_{L^2}
+ \frac{\tau^3}{2} ((\varphi')^3 \frac{b'}{\sqrt{b}} v, v)_{L^2}\\
& \quad - \tau(\varphi'' \nabla_{\theta}v, b \sqrt{b} \nabla_{\theta}v)_{L^2}
- \tau(\varphi' (\p_t(b \sqrt{b})) \nabla_{\theta}v, \nabla_{\theta} v)_{L^2}
- \tau (\frac{b'}{\sqrt{b}}\p_t v, \varphi' \p_t v)_{L^2}\\
& \quad -\tau \frac{(n-2)^2}{4}(\varphi'' \sqrt{b} v, v)_{L^2} - \frac{\tau}{2} \frac{(n-2)^2}{4}(\varphi' \frac{b'}{\sqrt{b}} v, v)_{L^2}.
\end{split}
\end{align}
Here, for ease of notation, we have abbreviated $b'(t,\theta):=\p_t b(t,\theta)$.
Next we consider the second contribution from \eqref{eq:mixed}. It turns into
\begin{align}
\label{eq:second_term}
\begin{split}
&- \tau \left(\sqrt{b} \p_t^2 v + \sqrt{b}\tau^2 (\varphi')^2 v + \D_{\Sigma}'v - \sqrt{b}\frac{(n-2)^2}{4}v, \varphi'' v \right)_{L^2} - \frac{\tau}{2}\left(\frac{b'}{\sqrt{b}} \p_t v, \varphi'' v\right)_{L^2}\\
& =  \frac{\tau}{2}(\frac{b'}{\sqrt{b}} \p_t v, \varphi'' v)_{L^2} + \tau (\sqrt{b} \p_t v, \varphi'' \p_t v)_{L^2}
+ \tau (\sqrt{b} \p_t v, \varphi''' v)_{L^2}
- \tau^3 (\sqrt{b} (\varphi')^2 \varphi'' v, v)_{L^2} \\
& \quad + (\nabla_{\theta} v, \varphi'' b \sqrt{b} \nabla_{\theta} v)_{L^2} -\frac{\tau}{2}(\frac{b'}{\sqrt{b}} \p_t v, \varphi'' v)_{L^2}  +\tau \frac{(n-2)^2}{4}(\varphi'' \sqrt{b} v, v)_{L^2}.
\end{split}
\end{align}
Combining the contributions from \eqref{eq:first_term}, \eqref{eq:second_term} leads to
\begin{align}
\label{eq:combined_12}
\begin{split}
&2 \tau (\sqrt{b} \varphi'' \p_t v, \p_t v)_{L^2}
+ 2 \tau^3(\sqrt{b} (\varphi')^2 \varphi'' v, v)_{L^2}\\
& - \frac{\tau}{2} (\varphi' \frac{b'}{\sqrt{b}} \p_t v, \p_t v)_{L^2}
+ \frac{\tau^3}{2} ((\varphi')^2 \frac{b'}{\sqrt{b}} v,v)_{L^2}
- \tau (\varphi' (\p_t(b \sqrt{b}))\nabla_{\theta} v, \nabla_{\theta} v)_{L^2} \\
&  + \tau (\sqrt{b} \p_t v, \varphi''' v)_{L^2}
- \frac{\tau}{2} \frac{(n-2)^2}{4}(\varphi' \frac{b'}{\sqrt{b}} v,v)_{L^2}.
\end{split}
\end{align}
The support assumption $\supp(v) \subset \{(t,\theta) \in (-\infty, t_0)\times \Sigma \}$ for a sufficiently small choice of $t_0$ combined with the explicit form of $\varphi$, and the bound
\begin{align*}
|b'(t)| \leq C e^{t},
\end{align*}
then a sufficiently large choice of $\tau_0$ allows us to estimate the contributions in \eqref{eq:combined_12} by
\begin{align}
\label{eq:combined_12a}
\begin{split}
&2 \tau (\sqrt{b} \varphi'' \p_t v, \p_t v)_{L^2}
+ 2 \tau^3(\sqrt{b} (\varphi')^2 \varphi'' v, v)_{L^2}\\
&- \frac{\tau}{2} (\varphi' \frac{b'}{\sqrt{b}} \p_t v, \p_t v)_{L^2}
+ \frac{\tau^3}{2} ((\varphi')^2 \frac{b'}{\sqrt{b}}v,v)_{L^2}
+ \tau (\sqrt{b} \p_t v, \varphi''' v)_{L^2}
\\
& \quad - \tau (\varphi' (\p_t(b \sqrt{b}))\nabla_{\theta} v, \nabla_{\theta} v)_{L^2} 
- \frac{\tau}{2} \frac{(n-2)^2}{4}(\varphi' \frac{b'}{\sqrt{b}} v,v)_{L^2}\\ 
& \geq 
 \tau (\sqrt{b} \varphi'' \p_t v, \p_t v)_{L^2}
+  \tau^3(\sqrt{b} (\varphi')^2 \varphi'' v, v)_{L^2} 
-  \tau (\varphi' (\p_t(b \sqrt{b}))\nabla_{\theta} v, \nabla_{\theta} v)_{L^2}.
\end{split}
\end{align}
This shows the control of tangential gradients and $L^2$ contributions (with an error involving the spherical derivatives). In Step 4 we show that the potentially negative contribution involving the spherical gradient can be absorbed into positive contributions and that we can upgrade \eqref{eq:combined_12a} to an estimate for the full gradient (including the spherical part of the gradient), for which we exploit the symmetric part of the operator.\\

\emph{Step 3: Estimates for the sublinear contributions.}
For the terms involving the sublinear potential, we argue similarly as in the proof of the model situation.
For ease of notation we define for $v \in \R$
\begin{align*}
&\tilde{f}_q(v)(t,\theta):= f_q(e^t \theta,v), \ \p_1 \tilde{f}_q(v)(t,\theta) = (\p_t \tilde{f}_q)|_{(t,\theta,v)}, \ \tilde{f}'_q(v)(t,\theta):= (\p_{v} \tilde{f}_q)|_{(t,\theta,v)},\\
&\tilde{F}_q(v)(t,\theta):= f_q(e^t \theta,v), \ \p_1 \tilde{F}_q(v)(t,\theta) = (\p_t \tilde{f}_q)|_{(t,\theta,v)}, \ \tilde{F}'_q(v)(t,\theta):= (\p_{v} \tilde{f}_q)|_{(t,\theta,v)}.
\end{align*}
With this notation at hand and using that
\begin{align*}
[\sqrt{b} h_q(\tilde{u}), -2\tau \varphi' \p_t - \tau \varphi'']
= [\sqrt{b} h_q(\tilde{u}), -2\tau \varphi' \p_t ],
\end{align*}
we compute
\begin{align}
\label{eq:non_lin_comm}
\begin{split}
-2\tau [\sqrt{b} h_q(\tilde{u}), \varphi' \p_t ]
& = 2\tau \varphi' \p_t (h_q(\tilde{u}) \sqrt{b})\\
& = \frac{2\tau e^{2t} \varphi' \sqrt{b}}{\Psi}\left( \frac{\tilde{f}_q'(\tilde{u})}{\tilde{u}} - \frac{\tilde{f}_q(\tilde{u})}{\tilde{u}^2} \right) \p_t \tilde{u}
+ \frac{4\tau e^{2t} \varphi' \sqrt{b}}{\Psi} \frac{\tilde{f}_q(\tilde{u})}{\tilde{u}}
 \\
& \quad + \frac{2\tau e^{2t} \varphi' \sqrt{b}}{\Psi}\frac{\p_1 \tilde{f}_q(\tilde{u})}{\tilde{u}} + \frac{\tau e^{2t} \varphi'}{\Psi} \frac{b'}{\sqrt{b}} \frac{\tilde{f}_q(\tilde{u})}{\tilde{u}} - 2\tau e^{2t} \varphi' \sqrt{b} \frac{\Psi'}{\Psi^2} \frac{\tilde{f}_q(\tilde{u})}{\tilde{u}}\\
& = \frac{2\tau e^{2t}\varphi' \sqrt{b}}{\Psi} \left( \frac{\tilde{f}_q'(\tilde{u})}{\tilde{u}} - \frac{\tilde{f}_q(\tilde{u})}{\tilde{u}^2} \right) \p_t \tilde{u} 
+ \frac{4\tau e^{2t} \varphi' \sqrt{b}}{\Psi} \frac{\tilde{f}_q(\tilde{u})}{\tilde{u}}
 + E_1.
\end{split}
\end{align}
We treat $E_1$ as an error, which we will discuss below, and thus first concentrate on the other contribution.
Using that $v= e^{\tau \varphi}e^{-\frac{n-2}{2}t} \tilde{u}$ and that the condition (F4) ensures the well-definedness of $\tilde{u}\tilde{f}'_q(\tilde{u})$, we infer
\begin{align}
\label{eq:non_lin_comm_1}
\begin{split}
&-2 \tau ([\sqrt{b} h_q(\tilde{u}), \varphi' \p_t] v, v)
= 2\tau \left( \frac{e^{2t} \varphi' v \sqrt{b}}{\Psi} \left(  \frac{\tilde{f}_q'(\tilde{u})}{\tilde{u}} - \frac{\tilde{f}_q(\tilde{u})}{\tilde{u}^2} \right) \p_t \tilde{u} , v \right) 
+ 4\tau \left(\frac{ e^{2t} \varphi' \sqrt{b}}{\Psi} \frac{\tilde{f}_q(\tilde{u})}{\tilde{u}}v,v \right) \\
& \quad + (E_1 v,v)\\
&= 2\tau \left(\frac{l \sqrt{b}}{\Psi}, (\tilde{u} \tilde{f}_q'(\tilde{u}) - \tilde{f}_q(\tilde{u}))\p_t \tilde{u}\right)
+ 4q\tau \left(\frac{ e^{2t} e^{2\tau \varphi}\varphi' \sqrt{b}}{\Psi} \frac{\tilde{F}_q(\tilde{u})}{\tilde{u}^2}, \tilde{u}^2 \right) 
 + (E_1 v,v)\\
& = 2 \tau \left( \frac{l \sqrt{b}}{\Psi}, \p_t (\tilde{u} \tilde{f}_q(\tilde{u})) - 2 \tilde{F}'_q(\tilde{u})\p_t \tilde{u} \right) 
+ 4q\tau \left(\frac{ e^{2t} e^{2\tau \varphi}\varphi' \sqrt{b}}{\Psi}, \tilde{F}_q(\tilde{u}) \right) 
+ (E_1 v,v)\\
&= 2 \tau \left(\frac{l \sqrt{b}}{\Psi}, \p_t (\tilde{u} \tilde{f}_q(\tilde{u})) - 2 \p_t( \tilde{F}_q(\tilde{u})) \right) + 4\tau(\sqrt{b} l, \p_1 \tilde{F}_q|_{\tilde{u}} ) + 4q\tau \left(\frac{ e^{2t} e^{2\tau \varphi}\varphi' \sqrt{b}}{\Psi}, \tilde{F}_q(\tilde{u}) \right)  + (E_1 v,v)\\
&= 2 \tau \left(\frac{l \sqrt{b}}{\Psi}, \p_t (\tilde{u} \tilde{f}_q(\tilde{u})) - 2 \p_t( \tilde{F}_q(\tilde{u})) \right)+ 4q\tau \left(\frac{ e^{2t} e^{2\tau \varphi}\varphi' \sqrt{b}}{\Psi}, \tilde{F}_q(\tilde{u}) \right)  + (E_2 v,v) + (E_1 v,v),
\end{split}
\end{align}
where we have set $l(t)= e^{(2-n)t} e^{2t} \varphi'(t) e^{2 \tau \varphi}$ and where we view $(E_2 v,v)$ as a controlled error.
We note that by our choice of the weight function $\varphi$
\begin{align}
\label{eq:l}
\begin{split}
l'(t) &=  e^{(2-n)t}  e^{2\tau \varphi}\left( \left(4- n\right)\varphi'(t) + \varphi''(t) + 2\tau (\varphi'(t))^2 \right)\\
&\geq 
e^{(2-n)t}  e^{2\tau \varphi}\left(  \varphi''(t) + \frac{3}{2} \tau (\varphi'(t))^2 \right)
\geq 0,
\end{split}
\end{align}
if $\tau \geq \tau_0>0$ is sufficiently large. Integrating the expression from \eqref{eq:non_lin_comm_1} by parts and using that $q\in [1,2)$, we thus further estimate
\begin{align}
\label{eq:comm_main}
\begin{split}
-2 \tau ([\sqrt{b} h_q(\tilde{u}), \varphi' \p_t] v, v)
&= - 2\tau \left( \frac{l' \sqrt{b}}{\Psi}, \tilde{u} \tilde{f}_q(\tilde{u})- 2\tilde{F}_q(\tilde{u}) \right)
+ 4\tau q \left(\frac{ e^{2t} e^{2\tau \varphi}\varphi' \sqrt{b}}{\Psi}, \tilde{F}_q(\tilde{u}) \right) \\
& \quad + \tau \left(\frac{b'}{\Psi \sqrt{b}}l, \tilde{u}\tilde{f}_q(\tilde{u})-2 \tilde{F}_q(\tilde{u})\right) -2 \tau \left( \frac{l \sqrt{b} \Psi'}{\Psi^2}, \tilde{u} \tilde{f}_q(\tilde{u}) - 2 \tilde{F}_q(\tilde{u}) \right) \\
& \quad +  ((E_1 + E_2)v,v) \\
& \geq 2(2-q)\tau \left( \frac{l' \sqrt{b}}{\Psi}, \tilde{F}_q(\tilde{u})\right) 
+ 4\tau \left(\frac{ e^{2t} e^{2\tau \varphi}\varphi' \sqrt{b}}{\Psi}, \tilde{F}_q(\tilde{u}) \right)\\
& \quad + \left((E_1 + E_2+ E_3)v,v \right) \\
& = 2(2-q)\tau (\Psi^{-1}e^{2\tau \varphi}e^{2t}\sqrt{b}(\varphi'' + \frac{3 \tau}{2} (\varphi')^2), \tilde{F}_q(\tilde{u})) \\
& \quad + 4\tau q \left(\frac{ e^{2t} e^{2\tau \varphi}\varphi' \sqrt{b}}{\Psi}, \tilde{F}_q(\tilde{u}) \right)
+ ((E_1+E_2+E_3)v,v)\\
& \geq 2(2-q)\tau (\Psi^{-1}e^{2\tau \varphi}e^{2t}\sqrt{b}(\varphi'' + \tau (\varphi')^2), \tilde{F}_q(\tilde{u})) \\
& \quad
+ ((E_1+E_2+E_3)v,v).
 \end{split}
\end{align}
We estimate the error terms $((E_1+E_2+E_3)v,v)$ and show that they are indeed of lower order, i.e. that they can be absorbed into the positive contributions on the right hand side of \eqref{eq:combined_12a}: To this end, we observe that by the assumption (F3) and by the definition of $F_q(x,s)$ as the antiderivative of $f_q(x,s)$ (c.f. the condition (F1))
\begin{align*}
\left| \frac{\p_1 \tilde{f}_q(\tilde{u})}{\tilde{u}} \right|
\leq \kappa_1 e^{t} \left| \frac{\tilde{f}_q(\tilde{u})}{\tilde{u}} \right|
\leq \kappa_1 e^{t} \left| \frac{\tilde{F}_q(\tilde{u})}{\tilde{u}^2} \right| .
\end{align*}
Thus,
\begin{align*}
\left| \left( v, e^{2t}\varphi' \sqrt{b} \Psi^{-1} \frac{\p_1 \tilde{f}_q(\tilde{u})}{\tilde{u}} v \right) \right|
&\leq \left| \left( v, e^{2t}|\varphi'| \sqrt{b} \Psi^{-1} \left|\frac{\p_1 \tilde{f}_q(\tilde{u})}{\tilde{u}} \right| v \right) \right|
\leq \left| \left( v, e^{2t}|\varphi'| \sqrt{b} \Psi^{-1} e^t \left|\frac{\tilde{F}_q(\tilde{u})}{\tilde{u}^2} \right| v \right) \right|\\
& \leq \kappa_1 \left| (|l|\sqrt{b} \Psi^{-1}, e^{t} |\tilde{F}_q(\tilde{u})| ) \right|.
\end{align*}
Similarly, 
\begin{align*}
\left| \left( e^{2t}\varphi' \frac{b'}{\sqrt{b} \Psi}  \frac{\tilde{f}_q(\tilde{u})}{\tilde{u}}v,v \right) \right|
&\leq \left| \left( e^{2t}|\varphi'|\left| \frac{b'}{\sqrt{b} \Psi} \right| \frac{|\tilde{f}_q(\tilde{u})\tilde{u}|}{|\tilde{u}|^2}v,v \right) \right|
\leq q\left| \left( |l||\varphi'| \left| \frac{b'}{\sqrt{b} \Psi} \right|, |\tilde{F}_q(\tilde{u})|e^{2\tau \varphi} \right) \right|\\
&\leq M q \left| \left( |l||\varphi'| e^{t}, |\tilde{F}_q(\tilde{u})| \right) \right|,
\end{align*}
and, since $|\Psi'|\leq C e^{t}$,
\begin{align*}
\left| \left( e^{2t} \varphi' \sqrt{b} \frac{\Psi'}{\Psi^2} \frac{\tilde{f}_q(\tilde{u})}{\tilde{u}} v, v \right) \right|
\leq C \left| \left( e^{2t} e^{2\tau \varphi} |\varphi'| \sqrt{b} \left|\frac{\Psi'}{\Psi^2}\right| \frac{\tilde{u}\tilde{f}_q(\tilde{u})}{\tilde{u}^2} v, v \right) \right|
\leq  C \left| \left( |l||\varphi'| \sqrt{b} e^{t},\tilde{F}_q(\tilde{u}) \right) \right|.
\end{align*}
As a consequence,
\begin{align}
\label{eq:error_1}
\begin{split}
|(v, E_1 v)| 
&\leq  2\tau \left|\left(e^{2t} v, \varphi' \sqrt{b} \frac{\p_1 \tilde{f}_q(\tilde{u})}{\tilde{u}} v \right) \right|
+ \tau \left| \left( e^{2t}\varphi' \frac{b'}{\sqrt{b}} \frac{\tilde{f}_q(\tilde{u})}{\tilde{u}} v, v \right) \right|+ 2\tau \left| \left( e^{2t}\varphi' \sqrt{b} \frac{\Psi'}{\Psi^2} \frac{\tilde{f}_q(\tilde{u})}{\tilde{u}} v, v \right) \right|\\
&\leq C(q,M,\kappa_1) \tau \left| \left( |l||\varphi'| e^{t}, \tilde{F}_q(\tilde{u}) \right) \right|.
\end{split}
\end{align}
With a similar reasoning we infer 
\begin{align}
\label{eq:error_2}
\begin{split}
&\left| \left(E_2 v, v \right) \right|
= 4\tau \left| \left( \sqrt{b} l, \p_1 \tilde{F}_q(\tilde{u}) \right) \right|
\leq 4 \tau \kappa_1 \left| \left( \sqrt{b} |l|, e^{t}\tilde{F}_q(\tilde{u})  \right) \right|.\\
&\left| \left( E_3 v,v \right) \right|
\leq \tau \left| \left( \frac{|b'|}{|\sqrt{b}|}|l| + 2 \frac{|l| \sqrt{b}|\Psi'|}{\Psi^2} , \tilde{u} \tilde{f}_q(\tilde{u})-2 \tilde{F}_q(\tilde{u}) \right) \right|
\leq (2-q)\tau \left| \left( \left| \frac{b'}{\sqrt{b}} \right||l| + 2 \frac{|l| \sqrt{b}|\Psi'|}{\Psi^2}, \tilde{F}_q(\tilde{u}) \right) \right|\\
&\qquad \qquad \leq  C(b)(2-q)\tau \left| \left( \sqrt{b} e^{t}|l|, \tilde{F}_q(\tilde{u}) \right) \right|.
\end{split}
\end{align}
Choosing $ r_0(q)>0$ (and thus also $t_0 = \ln(r_0)<0$) sufficiently small, we may hence absorb the error contributions from \eqref{eq:error_1}, \eqref{eq:error_2} into the positive term on the right hand side of \eqref{eq:comm_main}.
Using that by our assumptions $F_q(\tilde{u})\geq c |\tilde{u}|^{q}$ and choosing $\tau \geq \tau_0=\tau_0(q,n)>0$ sufficiently large, we then deduce that
\begin{align}
\label{eq:main_comm_2}
\begin{split}
-2\tau ([ h_q(\tilde{u}), \varphi' \p_t] v, v ) 
&\geq 2\tau \frac{2-q}{q} \int\limits_{\R \times S^{n-1}} e^{2t}\sqrt{b}\left(  \varphi''(t) + \tau (\varphi'(t))^2 \right) e^{2\tau \varphi} F_q(\tilde{u}) dt d\theta\\
& \qquad+ ((E_1+E_2+E_3)v,v)\\
& \geq  \tau \frac{2-q}{q} \int\limits_{\R \times S^{n-1}} e^{2t} \sqrt{b} \left(  \varphi''(t) + \tau (\varphi'(t))^2 \right) e^{2\tau \varphi} F_q(\tilde{u}) dt d\theta\\
& \geq  \tau \frac{2-q}{q} \int\limits_{\R \times S^{n-1}} e^{2t}\sqrt{b} \left(  \varphi''(t) + \tau (\varphi'(t))^2 \right)\max\{\kappa_2|\tilde{u}|^{q-2} v^2, e^{2\tau \varphi} F_q(\tilde{u}) \} dt d\theta.
\end{split}
\end{align}
After passing back to Cartesian coordinates, this concludes the argument for the derivation of the sublinear contribution.\\

\emph{Step 4: Upgrading the gradient estimate.} We explain the derivation of the full gradient estimates.
As in the corresponding estimate in Section \ref{sec:model}, this is based on the symmetric part of the operator. Testing it by $\tau c_0 \varphi'' v$ for a sufficiently small constant $c_0>0$, we infer
\begin{align}
\label{eq:sph_grad}
\begin{split}
c_0 \tau \||\varphi''|^{1/2} \nabla_{S^{n-1}} v\|_{L^2_{\vol}}^2
&\leq C c_0 \tau \left[|(Sv, \varphi'' v)_{L^2_{\vol}}|
+ \||\varphi''|^{1/2} \p_t v\|_{L^2_{\vol}}^2 \right.\\
& \quad \left. + C\tau^2 |(|\varphi''| |\varphi'|^2 v,v)_{L^2_{\vol}}|
+ |(h_q(\tilde{u})v,\varphi'' v)_{L^2_{\vol}}| \right]\\
& \leq \frac{1}{2}\|S v\|_{L^2_{\vol}}^2 + Cc_0 \tau^3 \||\varphi''|^{1/2} v\|_{L^2_{\vol}}^2 + C c_0 \tau \|e^t |\varphi''|^{1/2} e^{\tau \varphi} |F_q(\tilde{u})|^{1/2}\|_{L^2_{\vol}}^2 \\
&\stackrel{\eqref{eq:main_comm_2}, \eqref{eq:combined_12a}}{\leq}  \|S v\|_{L^2_{\vol}}^2 + ([S,A]v,v)_{L^2_{\vol}} + C\tau|(\varphi' \sqrt{b}^{-1}\p_t(b \sqrt{b}) \nabla_{\theta}v, \nabla_{\theta} v)_{L^2_{\vol}}|\\
&\leq C \|L v\|_{L^2_{\vol}}^2 + C\tau|(\varphi' \sqrt{b}^{-1}\p_t(b \sqrt{b}) \nabla_{\theta}v, \nabla_{\theta} v)_{L^2_{\vol}}|.
\end{split}
\end{align} 
Here $c_0>0$ is chosen so small that $Cc_0 \leq 1$ and where we used that 
\begin{align*}
&C c_0 \tau \||\varphi''|^{1/2} e^t e^{\tau \varphi} |F_q(\tilde{u})|^{1/2}\|_{L^2_{\vol}}^2\\
&\leq \tau \frac{2-q}{q} \int\limits_{\R \times S^{n-1}}\sqrt{b}e^t \left(  \varphi''(t) + \tau (\varphi'(t))^2 \right)\max\{\kappa_2|\tilde{u}|^{q-2} v^2, e^{2\tau \varphi} F_q(\tilde{u}) \} dt d\theta.
\end{align*}
if the support of $v$ is chosen sufficiently small and $\tau\geq \tau_0(n,q)$ is chosen sufficiently large (depending on $q$). 
Since 
\begin{align*}
|\p_t (\sqrt{b} b)| \leq c e^{t}|\sqrt{b}|,
\end{align*}
we may absorb the contribution $C\tau|(\varphi' \sqrt{b}^{-1}\p_t(b \sqrt{b}) \nabla_{\theta}v, \nabla_{\theta}v)_{L^2_{\vol}}|$ from the right hand side of \eqref{eq:sph_grad} into the left hand side of \eqref{eq:sph_grad} if $r_0>0$ is chosen appropriately small. Thus, we obtain
\begin{align*}
c_0 \tau \||\varphi''|^{1/2} \nabla_{S^{n-1}} v\|_{L^2_{\vol}}^2
&\leq C \|L v\|_{L^2_{\vol}}^2 .
\end{align*} 
\\

\emph{Step 5: Absorbing the error terms.}
Up to now we have proved the estimate
\begin{align}
\label{eq:intermediate}
\begin{split}
&\tau^{1/2} \||\varphi''|^{1/2} \nabla v\|_{L^2_{\vol}}
+ \tau^{3/2} \||\varphi''|^{1/2} v\|_{L^2_{\vol}} 
+ \tau \|e^t |\tilde{u}|^{\frac{q-2}{2}} v\|_{L^2_{\vol}}\\
&\leq C (\|L v\|_{L^2_{\vol}} + \|Ev\|_{L^2_{\vol}} ).
\end{split}
\end{align}
It hence remains to deal with the error contribution on the right hand side of \eqref{eq:intermediate}. Using the Lipschitz continuity of $a,b$ and $\Psi$ we can estimate
\begin{align*}
\|Ev\|_{L^2_{\vol}}= \tau \||\varphi' \p_t \ln(\sqrt{b})| v \|_{L^2_{\vol}}
\leq C \tau \|e^{t}  v\|_{L^2_{\vol}}.
\end{align*}
As before this can be absorbed into the left hand side of \eqref{eq:intermediate} (after possibly choosing $r_0>0$ even smaller and $\tau_0$ even larger). Returning to Cartesian coordinates, then concludes the proof of the Carleman estimate for the operator $Lu=\D_{a^{ij}} u + \frac{\hat{f}_q(u)}{\Psi}$. Using the equivalence from Lemma \ref{lem:geo}, we then also infer a Carleman estimate for the operator $\p_i a^{ij}\p_j u + \hat{f}_q(u)$. This involves lower order errors of the type $R$ from \eqref{eq:LB}, but as outlined in the previous error estimates, these can be absorbed into the left hand side of the Carleman estimate. Hence, we arrive at the desired result of Theorem \ref{prop:varmet}.
\end{proof}

\bibliographystyle{alpha}
\bibliography{citationsHT}

\end{document}